\documentclass[a4paper,12pt]{article}

\usepackage[english]{babel}
\usepackage{amsmath,amsfonts,amssymb,amsthm,bbm}

\usepackage[top=2cm, bottom=2cm, left=2.5cm, right=2.5cm]{geometry}

\usepackage{indentfirst}
\usepackage[T1]{fontenc} 

\usepackage{graphicx}
\usepackage{float}

\usepackage{epsfig}
\usepackage{epstopdf}

\newtheorem{theorem}{Theorem}

\newtheorem{remark}{Remark}
\newtheorem{corollary}{Corollary}

\numberwithin{equation}{section}

\newcommand{\EE}{\mathbb F}
\renewcommand{\O}{\mathbb O}
\renewcommand{\H}{\mathbb H}
\newcommand{\C}{\mathbb C}
\newcommand{\R}{\mathbb R}

\newcommand{\N}{\mathbb N}

\newcommand*{\abs}[1]{\left\lvert{#1}\right\rvert} 
\newcommand*{\norm}[1]{\left\lVert{#1}\right\rVert} 
\newcommand*{\conj}[1]{#1^*} 
\newcommand*{\innerL}[2]{\left({#1},{#2}\right)}

\DeclareMathOperator*{\Ima}{\mathrm{Im}}
\DeclareMathOperator*{\Rea}{\mathrm{Re}}

\newcommand{\e}{\mathbf e}

\renewcommand{\i}{\mathbf i}
\renewcommand{\j}{\mathbf j}
\newcommand{\kk}{\mathbf k}

\newcommand{\x}{\mathbf x}
\newcommand{\y}{\mathbf y}
\newcommand{\f}{\mathbf f}
\newcommand{\de}{\mathrm d}

\newcommand*{\FourierX}[2]{\mathcal{F}_{\mathrm{#1}}\left\{#2\right\}}
\newcommand*{\IFourierX}[2]{\mathcal{F}^{-1}_{\mathrm{#1}}\left\{#2\right\}}

\usepackage{array}
\newcolumntype{L}[1]{>{\raggedright\let\newline\\\arraybackslash\hspace{0pt}}m{#1}}
\newcolumntype{C}[1]{>{\centering\let\newline\\\arraybackslash\hspace{0pt}}m{#1}}
\newcolumntype{R}[1]{>{\raggedleft\let\newline\\\arraybackslash\hspace{0pt}}m{#1}}

\begin{document}

\begin{center}
\Large
\textbf{A Generalization of the Octonion Fourier Transform \\ to 3-D Octonion-Valued Signals \\ -- Properties and Possible Applications \\ to 3-D LTI Partial Differential Systems}\\
\end{center}


\begin{center}
\begin{tabular}{cc}
	\textbf{{\L}ukasz B{\l}aszczyk\textsuperscript} \\
	l.blaszczyk@mini.pw.edu.pl 
\medskip \\
	Faculty of Mathematics and Information Science \\
	Warsaw University of Technology \\
	ul. Koszykowa 75, 00-662 Warszawa, Poland
\end{tabular}
\end{center}

\bigskip\noindent
\textbf{Keywords}: octonion Fourier transform, Cayley-Dickson numbers, hypercomplex algebras, multidimensional linear time-invariant systems.

\bigskip
\normalsize
\begin{abstract}
The paper is devoted to the development of the octonion Fourier transform (OFT) theory initiated in 2011 in articles by Hahn and Snopek. It is also a continuation and generalization of earlier work by B{\l}aszczyk and Snopek, where they proved few essential properties of the OFT of real-valued functions, e.g. symmetry properties. The results of this article focus on proving that the OFT is well-defined for octonion-valued functions and almost all well-known properties of classical (complex) Fourier transform (e.g. argument scaling, modulation and shift theorems) have their direct equivalents in octonion setup. Those theorems, illustrated with some examples, lead to the generalization of another result presented in earlier work, i.e. Parseval and Plancherel Theorems, important from the signal and system processing point of view. Moreover, results presented in this paper associate the OFT with 3-D LTI systems of linear PDEs with constant coefficients. Properties of the OFT in context of signal-domain operations such as derivation and convolution of $\mathbb{R}$-valued functions will be stated. There are known results for the QFT, but they use the notion of other hypercomplex algebra, i.e. double-complex numbers. Considerations presented here require defining other higher-order hypercomplex structure, i.e. quadruple-complex numbers. This hypercomplex generalization of the Fourier transformation provides an excellent tool for the analysis of 3-D LTI systems.
\end{abstract}

\section{Introduction}
\label{sec:1}

Fourier analysis is one of the fundamental tools in signal and image processing. Fourier series and Fourier transform enable us to look at the concept of signal in a~dual manner -- by studying its properties in the time domain (or in the space domain in case of images), where it is represented by amplitudes of the samples (or pixels), or by investigating it in the frequency domain, where the signal can be represented by the infinite sums of complex harmonic functions, each with different frequency and amplitude \cite{Allen2004}.
\medskip

The classical signal theory deals with real- or complex-valued time series (or images). However, in some practical applications, signals are represented by more abstract structures, e.g. hypercomplex algebras \cite{Ell2014,Grigoryan2018,HahnSnopek2016,Snopek2015}. A classic example is the use of them in the processing of color images (where there are at least three color components)~\cite{Ell2014,Grigoryan2018}, but also in the analysis of multispectral data (e.g. in satellite images where not only visible light is recorded, but also other frequency ranges)~\cite{Lazendic2018,Lazendic2018a}. Quaternions and octonions deserve special attention in this considerations. They are examples of Cayley-Dickson (C-D) algebras~\cite{Dickson}. C-D algebras are defined by a recursive procedure, so-called Cayley-Dickson construction. They are algebras of the order $2^N$ ($N\in\N$) over the field of real numbers~$\R$. Each C-D algebra is created from the previous one and contains all previous algebras as proper sub-algebras.
\medskip

Recently, hypercomplex algebras drew scientists' attention due to their numerous applications, among others in the study of neural networks~\cite{Popa2016,Popa2018,Wu2019}, in the analysis of color and multispectral images~\cite{Ell2014,Gao2014,Gomes2017,Grigoryan2018,Lazendic2018,Lazendic2018a,Sheng2018}, in biomedical signal processing~\cite{Delsuc1988,Klco2017}, in fluid mechanics~\cite{Demir2016} or in general signal processing~\cite{HahnSnopek2016,Snopek2015,Wang2017}. Quaternions may be used in few different ways -- to describe a~vector-valued signal (with three or four coordinates) of one variable, i.e. \begin{align*}
u(t) = u_{0}(t) + u_{1}(t) \cdot\i + u_{2}(t) \cdot\j + u_{3}(t) \cdot\kk,
\qquad u_{0}, u_{1}, u_{2}, u_{3}\colon \R\to\R,
\end{align*} or to analyse a~scalar or vector-valued signal of two variables, i.e. $u\colon\R^2\to\R$ or \linebreak $u\colon\R^2\to\O$. The basic tool in the second approach is the quaternion Fourier transform (QFT) \cite{Bulow2001}: \begin{equation} \label{eq:1_1}
U_{\mathrm{QFT}}(f_1,f_2) = \int_{\R}\int_{\R} u(t_1,t_2)e^{-2\pi\i f_1 t_1}e^{-2\pi\j f_2 t_2}\,\de t_1\de t_2.
\end{equation} It allows us (in contrast to the classical two-dimmensional Fourier transform) to analyse two dimensions of the sampling grid independently. Each time-like dimension can be associated with a~different dimension of the four-dimensional quaternion space, while the complex transform mixes those two dimensions. It also allows us to study some symmetries present in certain signals (images), what was impossible before \cite{Ell2014}. Similarly, octonions are used to describe scalar or vector-valued signals of one or three variables.
\medskip

In the last few years some generalizations of the Fourier transform (defined as in \eqref{eq:1_1}) to the octonion and higher-order algebras appeared in the literature~\cite{HahnSnopek2011a,Snopek2009,Snopek2011,Snopek2012,Snopek2015}. They are defined on the basis of the Cayley-Dickson algebras and called the Cayley-Dickson Fourier transforms. The main goal of this paper is further development of such generalization based on the Cayley-Dickson algebra of order 8 (octonions). Analysis of the current state of knowledge on applications of octonions in the signal processing shows some areas previously unexplored or requiring thorough theoretical and experimental studies, although some gaps have recently been filled~\cite{BlaszczykECMI2018,BlaszczykEUSIPCO2018,BlaszczykSnopek2017,Lian2019}.  
\medskip

Properties of the quaternion Fourier transform (defined by \eqref{eq:1_1}) are well studied in the literature and it is fairly easy to notice some analogies to the properties of classical (complex) Fourier transform of functions of two variables \cite{Ell1992}. They enable us to use the Fourier transform in the analysis of some two-dimensional linear time-invariant systems described by systems of partial differential equations with constant coefficients \cite{Ell1993}. In our previous investigations \cite{BlaszczykSnopek2017} we were able to show that the OFT is well defined for scalar (real-valued) functions of three variables (i.e. we proved the inverse transform theorem). In our research we also derived some properties of the OFT, analogous to the properties of the classical (complex) and quaternion Fourier transform, e.g. symmetry properties (analogue to the Hermitian symmetry properties), shift theorem, Plancherel and Parseval theorems, and Wiener-Khintchine theorem. Proofs of the those theorems were based on the previous research of Hahn and Snopek, who used the fact that real--valued functions can be expressed as a~sum of components of different parity \cite{HahnSnopek2011a}. Despite these works, the state of modern knowledge about octonion Fourier transform is negligible and requires a~thorough extension. This seems important especially in the context of new applications that have appeared in recent years (and which we described earlier in this section) -- there is a tendency to describe the results of practical experiments, but without adequate theoretical justification.
\medskip

Some of the results presented in this paper have been signaled in earlier works~\cite{BlaszczykECMI2018,BlaszczykEUSIPCO2018}, here we present a~broader view of these issues and give details of the proofs. We also provide some new results, mainly regarding the use of classical transformation techniques for calculating the OFT, as well as regarding the differentiation of the octonion transform.
\medskip

The paper is organized as follows. 
In Section~\ref{sec:2} we recall the octonion algebra, some of its basic properties and the definition of the octonion Fourier transform, as well as the proof of its well-posedness. We also introduce the notion of the quadruple-complex algebra. 
In Section~\ref{sec:3} we focus on deriving some important properties of the OFT, e.g.~argument scaling, modulation and shift theorems, relationship between the OFT of a~function and the OFT of its partial derivative, differentiation of the OFT and the convolution theorem. 
Those considerations lead to some remarks on~applying the OFT to the analysis of 3-D linear time-invariant systems in Section~\ref{sec:5}, which also show the advantages of using OFT over classical transformation. 
The paper is conculed in Section~\ref{sec:6} with a~short discussion of obtained results.

\section{Basic definitions}
\label{sec:2}

In this section, we introduce the definitions and the theorems necessary to present the main results of this work regarding the further properties of the octonion Fourier transform.

\subsection{Algebra of octonions}
An octonion $o\in\O$ is defined, according to the Cayley-Dickson construction, as the ordered pair of quaternions \cite{Dickson}: \begin{align*}
o=(q_0,q_1),\quad \text{where}\quad
q_0 = r_0 + r_1\, \e_1 + r_2\, \e_2 + r_3\, \e_3,\;\;\; 
q_1 = r_4 + r_5\, \e_1 + r_6\, \e_2 + r_7\, \e_3\in\H
\end{align*} (we denote the quaternion imaginary units as $\e_1$, $\e_2$ and $\e_3$ instead of traditional $\i$, $\j$ and $\kk$). Rules of octonion multiplication are given by the general Cayley-Dickson formula \begin{align} \label{eq:2_0}
(q_0, q_1)\cdot (p_0, p_1) = (q_0\cdot p_0 - p_1^*\cdot q_1,\;\;  p_1\cdot q_0 + q_1\cdot p_0^*), \qquad q_0,q_1,p_0,p_1\in\H,
\end{align} where multiplication of quaternions is defined as in~\cite{Rodman2014} (it can be defined also by the formula~\eqref{eq:2_0} if we treat a~quaternion as an ordered pair of complex numbers) and ${}^*$ is quaternion conjugate. Applying those rules of multiplication (which can be presented in the form of Tab. \ref{tab:2_1}) we get four new imaginary units and octonions can be writen as \begin{align*}
o &= \underbrace{r_0 + r_1 \e_1 + r_2 \e_2 + r_3 \e_3}_{=q_0} + (\underbrace{r_4 + r_5 \e_1 + r_6 \e_2 + r_7 \e_3}_{=q_1})\cdot \e_4 \nonumber \\
&= r_0 + r_1 \e_1 + r_2 \e_2 + r_3 \e_3 + r_4 \e_4 + r_5 \e_5 + r_6 \e_6 + r_7 \e_7. 
\end{align*} Number $r_0\in\R$ is called the real part of $o$ (denoted as $\Rea o$) and the pure imaginary octonion $r_1\, \e_1 + r_2\, \e_2 + \ldots + r_7\, \e_7$ is called the imaginary part of $o$ (and denoted as $\Ima o$). Octonions form a \emph{non-associative} and a \emph{non-commutative} algebra, which means that in general, for $o_1,\, o_2,\, o_3\in\O$ \begin{align*}
&(o_1\cdot o_2)\cdot o_3 \neq o_1\cdot (o_2\cdot o_3), \qquad o_1\cdot o_2 \neq o_2\cdot o_1.
\end{align*} On the other hand, it is true that for any $o_1,\,o_2\in\O$ we have \begin{align*}
& (o_1 \cdot o_2)^* = o_2^* \cdot o_1^*, 
\end{align*} where $^*$ is the octonion conjugate, i.e. \begin{align*}
o^* = r_0 - r_1 \e_1 - r_2 \e_2 - r_3 \e_3 - r_4 \e_4 - r_5 \e_5 - r_6 \e_6 - r_7 \e_7.
\end{align*} As in case of complex numbers or quaternions, octonion conjugation is linear and we have $o^{**}=o$, which means that it is an involution. For any $o_1,\,o_2\in\O$ we also have that \begin{align}
 o_1 \cdot (o_1 \cdot o_2) = (o_1 \cdot o_1) \cdot o_2, \qquad\qquad
 (o_1 \cdot o_2) \cdot o_2 = o_1 \cdot (o_2 \cdot o_2) \label{eq:2_3} 
\end{align} and \begin{align}
o_1\cdot(o_2\cdot o_1) = (o_1\cdot o_2)\cdot o_1, \label{eq:2_3a} 
\end{align} which means that the algebra of octonions is \emph{alternative} (equations~\eqref{eq:2_3}) and \emph{flexible} (equation~\eqref{eq:2_3a}). 
\medskip

\renewcommand{\arraystretch}{1.2}
\begin{table}[t]
{\footnotesize\begin{center}
\begin{tabular}{|C{7mm}||C{7mm}|C{7mm}|C{7mm}|C{7mm}|C{7mm}|C{7mm}|C{7mm}|C{7mm}|}
\hline
$\cdot$ &    $1$ &  $\e_1$ &  $\e_2$ &  $\e_3$ &  $\e_4$ &  $\e_5$ &  $\e_6$ &  $\e_7$ \\
\hline 
\hline
    $1$ &    $1$ &  $\e_1$ &  $\e_2$ &  $\e_3$ &  $\e_4$ &  $\e_5$ &  $\e_6$ &  $\e_7$ \\
\hline
  $\e_1$ &  $\e_1$ &   $-1$ &  $\e_3$ & $-\e_2$ &  $\e_5$ & $-\e_4$ & $-\e_7$ &  $\e_6$ \\
\hline
  $\e_2$ &  $\e_2$ & $-\e_3$ &   $-1$ &  $\e_1$ &  $\e_6$ &  $\e_7$ & $-\e_4$ & $-\e_5$ \\
\hline
  $\e_3$ &  $\e_3$ &  $\e_2$ & $-\e_1$ &   $-1$ &  $\e_7$ & $-\e_6$ &  $\e_5$ & $-\e_4$ \\
\hline
  $\e_4$ &  $\e_4$ & $-\e_5$ & $-\e_6$ & $-\e_7$ &   $-1$ &  $\e_1$ &  $\e_2$ &  $\e_3$ \\
\hline
  $\e_5$ &  $\e_5$ &  $\e_4$ & $-\e_7$ &  $\e_6$ & $-\e_1$ &   $-1$ & $-\e_3$ &  $\e_2$ \\
\hline
  $\e_6$ &  $\e_6$ &  $\e_7$ &  $\e_4$ & $-\e_5$ & $-\e_2$ &  $\e_3$ &   $-1$ & $-\e_1$ \\
\hline
  $\e_7$ &  $\e_7$ & $-\e_6$ &  $\e_5$ &  $\e_4$ & $-\e_3$ & $-\e_2$ &  $\e_1$ &   $-1$ \\
\hline
\end{tabular}\end{center}}
\vspace{2mm}
\caption{Multiplication rules in octonion algebra.}
\label{tab:2_1}
\end{table}

In complex numbers we have the trigonometric form of a~number and in octonion algebra we can define a~similar formula for any nonzero octonion $o\in\O$: \begin{align} \label{eq:2_3p}
o = \abs{o}\cdot(\cos\theta + \boldsymbol{\mu}\cdot\sin\theta),
\end{align} where $\abs{o} = \sqrt{o\cdot o^*}$ is octonion norm, $\boldsymbol{\mu} = \frac{\Ima o}{\abs{\Ima o}}$ is pure imaginary octonion and $\theta\in\R$ is the solution of the system of equations \begin{align*}
\cos\theta = \frac{\Rea o}{\abs{o}},\qquad \sin\theta = \frac{\abs{\Ima o}}{\abs{o}}.
\end{align*} To formulate the exponential form of an octonion, we have to define octonion exponential function first. Similarly as for the complex numbers and quaternions~\cite{Rodman2014}, we use the infinite series. For any $o\in\O$, \begin{align*}
e^o = \exp(o) := \sum\limits_{k=0}^{\infty} \frac{o^k}{k!}.
\end{align*} It can be shown that if we denote $\mathbf{o} = \Ima o$, then \begin{align*}
e^o = e^{\Rea o} \left(\cos \abs{\mathbf{o}} + \frac{\mathbf{o}}{\abs{\mathbf{o}}} \sin\abs{\mathbf{o}}\right).
\end{align*} One should keep in mind that the fundamental multiplicative identity is in general not valid for octonions. For any $o_1,o_2\in\O$ we have \begin{align*}
e^{o_1 + o_2} = e^{o_1}\cdot e^{o_2}\qquad\text{if and only if}\qquad o_1 \cdot o_2 = o_2\cdot o_1,
\end{align*} which follows from the fact that the octonion multiplication is non-commutative.
\medskip

From the above considerations it immediately follows that the exponential form of an octonion $o\in\O$, $o\neq 0$, can be defined as \begin{align*}
o = \abs{o}\cdot e^{\theta\boldsymbol{\mu}},
\end{align*} where $\theta$ and $\boldsymbol{\mu}$ are defined as in~\eqref{eq:2_3p}. We can also generalize well-known formulas for trigonometric functions, i.e. for any $\alpha\in\R$ we have that \begin{align} \label{eq:2_4}
\cos\alpha = \frac{1}{2}\left(e^{\boldsymbol{\mu}\alpha} + e^{-\boldsymbol{\mu}\alpha}\right), \qquad \sin\alpha = \frac{1}{2\boldsymbol{\mu}}\left(e^{\boldsymbol{\mu}\alpha} - e^{-\boldsymbol{\mu}\alpha}\right),
\end{align} where $\boldsymbol{\mu}$ is any octonion such that $\abs{\boldsymbol{\mu}}=1$ and $\Rea\boldsymbol{\mu}=0$ (i.e. $\boldsymbol{\mu}$ is pure unitary octonion). It should be noted that every non-zero octonion is invertible and for pure unitary octonions $\boldsymbol{\mu}$ we have $\boldsymbol{\mu}^{-1} = -\boldsymbol{\mu}$.

\subsection{Algebra of quadruple-complex numbers}

Many formulas presented in Section~\ref{sec:3}, concerning the Fourier transforms, are quite complicated due to the fact that octionion multiplication is non-associative and non-commutative. Inspired by~\cite{Ell1993}, we introduce \textit{the algebra of quadruple-complex numbers}, which will allow us to reformulate all presented properties and show them in a simpler form, very similar to those well-known for classic Fourier transform. 
\medskip

Like octonions, we will define the algebra of order $8$ over the field of real numbers and each element of this algebra will be identified with the $8$-tuple of real numbers, i.e. \begin{align*}
p = p_0+p_1\e_1+p_2\e_2+p_3\e_3+p_4\e_4+p_5\e_5+p_6\e_6+p_7\e_7\in\EE, \quad p_0,\ldots,p_7\in\R.
\end{align*} Addition in~$\EE$ is defined in a~classical way -- element-wise. Before we define the multiplication, recall that in Cayley-Dickson construction, every octonion can be writen as an ordered pair of quaternions. We are going now one step further and rewrite an octonion as a~quadruple of complex numbers: \begin{align}
p &= (p_0+p_1\e_1) + (p_2+p_3\e_1)\e_2 + (p_4+p_5\e_1)\e_4 + (p_6+p_7\e_1)\e_2\e_4 \nonumber\\
&= s_0 + s_1\e_2 + s_2\e_4 + s_3\e_2\e_4, \label{eq:7_7}
\end{align} where $s_0,\ldots,s_3\in\C$ and multiplication is done from left to right.
\medskip

We will identify each element of~$\EE$ with a~quadruple of complex numbers $(s_0,s_1,s_2,s_3)$. Every element of $\EE$ will correspond to exactly one octonion defined by~\eqref{eq:7_7}. Multiplication $\odot$ is given by the formula \begin{align*}
(s_0,s_1,s_2,s_3)\odot(t_0,t_1,t_2,t_3) 
&= ( s_0t_0-s_1t_1-s_2t_2+s_3t_3,\quad s_0t_1+s_1t_0-s_2t_3-s_3t_2, \nonumber\\
&\phantom{=}\quad  s_0t_2+s_2t_0-s_1t_3-s_3t_1,\quad s_0t_3+s_3t_0+s_1t_2+s_2t_1) 
\end{align*} for every $(s_0,s_1,s_2,s_3), (t_0,t_1,t_2,t_3)\in\EE$. After straightforward computation we get the multiplication rules table, like we have in case of octonions (Tab.~\ref{tab:7_1}). We can see that imaginary units in $\EE$ don't follow the same rules that applied to octonions, i.e. \begin{align*}
\e_1\odot\e_1=\e_2\odot\e_2=-\e_3\odot\e_3=\e_4\odot\e_4=-\e_5\odot\e_5=-\e_6\odot\e_6=\e_7\odot\e_7=-1.
\end{align*} There is a~similarity to double-complex numbers \cite{Podwojne}, which have been used in the analysis of 2-D systems~\cite{Ell1993} and (though not so named) in hypercomplex representation of 2D nuclear magnetic resonance spectra~\cite{Delsuc1988}.
\medskip

\renewcommand{\arraystretch}{1.2}
\begin{table}[t]
{\footnotesize\begin{center}
\begin{tabular}{|C{7mm}||C{7mm}|C{7mm}|C{7mm}|C{7mm}|C{7mm}|C{7mm}|C{7mm}|C{7mm}|}
\hline
 $\odot$ &    $1$ &  $\e_1$ &  $\e_2$ &  $\e_3$ &  $\e_4$ &  $\e_5$ &  $\e_6$ &  $\e_7$ \\
\hline\hline
     $1$ &    $1$ &  $\e_1$ &  $\e_2$ &  $\e_3$ &  $\e_4$ &  $\e_5$ &  $\e_6$ &  $\e_7$ \\
\hline
   $\e_1$ &  $\e_1$ &   $-1$ &  $\e_3$ & $-\e_2$ &  $\e_5$ & $-\e_4$ &  $\e_7$ & $-\e_6$ \\
\hline
   $\e_2$ &  $\e_2$ &  $\e_3$ &   $-1$ & $-\e_1$ &  $\e_6$ &  $\e_7$ & $-\e_4$ & $-\e_5$ \\
\hline
   $\e_3$ &  $\e_3$ & $-\e_2$ & $-\e_1$ &    $1$ &  $\e_7$ & $-\e_6$ & $-\e_5$ &  $\e_4$ \\
\hline
   $\e_4$ &  $\e_4$ &  $\e_5$ &  $\e_6$ &  $\e_7$ &   $-1$ & $-\e_1$ & $-\e_2$ & $-\e_3$ \\
\hline
   $\e_5$ &  $\e_5$ & $-\e_4$ &  $\e_7$ & $-\e_6$ & $-\e_1$ &    $1$ & $-\e_3$ &  $\e_2$ \\
\hline
   $\e_6$ &  $\e_6$ &  $\e_7$ & $-\e_4$ & $-\e_5$ & $-\e_2$ & $-\e_3$ &    $1$ &  $\e_1$ \\
\hline
   $\e_7$ &  $\e_7$ & $-\e_6$ & $-\e_5$ &  $\e_4$ & $-\e_3$ &  $\e_2$ &  $\e_1$ &   $-1$ \\
\hline
\end{tabular}\end{center}}
\vspace{2mm}
\caption{Multiplication rules in $\EE$.}
\label{tab:7_1}
\end{table}

The multiplication in~$\mathbb{F}$ is commutative and associative. One can also show that there are no zero divisors in $\EE$ (i.e. if $s\,\odot\, t=0$, then $s=0$ or $t=0$), however not every non-zero element of $\EE$ has its $\odot$-inverse. If inverse of $(s_0,s_1,s_2,s_3)$ exists then it is the only one and is equal to \begin{align}
(s_0,s_1,s_2,s_3)^{-1} &= \frac{1}{\delta}\bigg(s_0(s_0^2+s_1^2+s_2^2-s_3^2)+2s_1s_2s_3,\;\; -s_1(s_0^2+s_1^2-s_2^2+s_3^2)-2s_0s_2s_3, \nonumber\\
&\!\!\!\!\!\!\!\! -s_2(s_0^2-s_1^2+s_2^2+s_3^2)-2s_0s_1s_3,\;\; s_3(-s_0^2+s_1^2+s_2^2+s_3^2)+2s_0s_1s_2\bigg), \label{eq:7_9}
\end{align} where \begin{align*}
\delta = \big((s_0-s_3)^2+(s_1+s_2)^2\big)\big((s_0+s_3)^2+(s_1-s_2)^2\big).
\end{align*} Elements of $\EE$ for which $\delta=0$ (e.g. $(1,0,0,\pm1)=1\pm\e_6\in\mathbb{F}$) have no $\odot$-inverse. One can easily notice that the equation~\eqref{eq:7_9} is similar to formula (3.4) in~\cite{Ell1993} for double-complex numbers, but every numer in (3.4) was a~real number. In~\eqref{eq:7_9} we have (in the general case) complex numbers.
\medskip

%

\subsection{Octonion Fourier transform}

Definition of the octonion Fourier transform (OFT) of the real-valued function of three variables was introduced in~\cite{Snopek2009} and used in later publications concerning theory of hypercomplex analytic functions~\cite{HahnSnopek2011a,Snopek2009,Snopek2011,Snopek2012,Snopek2015}. In~\cite{BlaszczykSnopek2017} we proved that the OFT of real-valued function is well-defined and has some interesting properties (such as the analogue of the Hermitian symetry). In~\cite{BlaszczykEUSIPCO2018} we stated that the inverse transform formula is correct for the octonion-valued functions and we presented the sketch of the proof. In the further part of this section we will present previously omitted details. 
\medskip

Consider the octonion-valued function of three variables $u\colon\R^3\to\O$, i.e. \begin{align*}
u(\x)=u_0(\x)+u_1(\x)\e_1+\ldots+u_7(\x)\e_7, \quad
u_i\colon\R^3\to\R,\,i=0,\ldots,7, \quad
\x = (x_1,x_2,x_3).
\end{align*} The octonion Fourier Transform of the integrable (in Lebesgue sense) function $u$ is given by the formula \begin{equation} \label{eq:oct_def}
U_{\mathrm{OFT}}(\f) = \int_{\R^3} u(\x) e^{-\e_1 2\pi f_1 x_1} e^{-\e_2 2\pi f_2 x_2} e^{-\e_4 2\pi f_3 x_3}\,\de\x.
\end{equation} Recall that the octonion algebra is non-associative, so it is necessary to note that the multiplication in the above integrals is done from left to right. As we already explained in~\cite{BlaszczykSnopek2017}, choice and order of imaginary units in the exponents is not accidental. In order for the integral \eqref{eq:oct_def} to exist, it is be necessary for the function to be at least integrable. In general, conditions of existence of the OFT are the same as for the classical (complex) Fourier transform. It is worth noting here the advantage of using octonion transformation over the use of classical transformation. Unlike the classical Fourier transformation, the OFT kernel is no longer a one-dimensional function, but it changes independently in three orthogonal directions. This is an analogous observation to that which was made in the case of the quaternion Fourier transform~\cite{bulow,Bulow2001}.

In this section, we will focus on the invertibility of the OFT. For the special case of the real-valued functions we proved the following theorem in~\cite{BlaszczykSnopek2017}. Here we prove the general version of the theorem, where the tested function has octonion values.
\medskip

\begin{theorem} \label{th:3_1}
Let $u\colon \R^3\to\O$ be continuous and let both $u$ and its OFT be integrable (in Lebesgue sense). Then for all $\x\in\R^3$ we have \begin{equation*}
u(\x) = \int_{\R^3} U_{\mathrm{OFT}}(\f) e^{\e_4 2\pi f_3 x_3} e^{\e_2 2\pi f_2 x_2} e^{\e_1 2\pi f_1 x_1} \,\de\f
\end{equation*} (where multiplication is performed from left to right).
\end{theorem}
\medskip

The assumptions given above are quite strong and in many cases can be mitigated. A~number of other conditions are known in the literature for the classic Fourier transform to be invertible and the equivalent of the above formula occurs~\cite{Allen2004,Duoandikoetxea2001}. Then we usually deal with equality almost everywhere, or the integral is understood in the sense of the principal value. In the case of an octonion transformation, these conditions are identical and detailed considerations are left to the reader. The abovementioned result follows from Fourier Integral Theorem~\cite{Duoandikoetxea2001}, which we state under the same assumptions as in Theorem~\ref{th:3_1}. 
\medskip

\begin{theorem} \label{th:3_2}
Let $u\colon \R^n\to\R$ be continuous and let both $u$ and its OFT be integrable (in Lebesgue sense). Then \begin{equation*}
u(\x) = \int_{\R^n} \int_{\R^n} u(\y) e^{2\pi \i \, \f\cdot(\x-\y)}\,\de\y \,\de\f,
\end{equation*} where $\i=\e_1$ is complex imaginary unit, $\x=(x_1,\ldots,x_n)$, $\y=(y_1,\ldots,y_n)$, $\f=(f_1,\ldots,f_n)$ and $\cdot$ is classic scalar product.
\end{theorem}

\begin{proof} [of Theorem~\ref{th:3_1}]
We need to prove the following equation \begin{align*}
u(\x) &= \int_{\R^3} \int_{\R^3} u(\y) \cdot e^{-\e_1 2\pi f_1 y_1} \cdot e^{-\e_2 2\pi f_2 y_2} \cdot e^{-\e_4 2\pi f_3 y_3} \cdot e^{\e_4 2\pi f_3 x_3} \cdot e^{\e_2 2\pi f_2 x_2} \cdot e^{\e_1 2\pi f_1 x_1} \,\de\y\de\f,
\end{align*} where octonion multiplication is done from left to right. The first step is to rewrite the function~$u$ as a~sum $u=u_0+u_1\e_1+\ldots+u_7\e_7$ and use the distributive law on the algebra of octonions. It follows that the claim of the theorem is equivalent to the system of equations \begin{align}
u_0(\x) &= \int_{\R^3} \int_{\R^3} u_0(\y) \cdot e^{-\e_1 2\pi f_1 y_1} \cdot e^{-\e_2 2\pi f_2 y_2} \cdot e^{-\e_4 2\pi f_3 y_3} \nonumber\\
&\qquad\qquad\qquad \cdot e^{\e_4 2\pi f_3 x_3} \cdot e^{\e_2 2\pi f_2 x_2} \cdot e^{\e_1 2\pi f_1 x_1} \,\de\y\,\de\f, \label{eq:3_5} \\
u_i(\x)\e_i &= \int_{\R^3} \int_{\R^3} u_i(\y)\e_i \cdot e^{-\e_1 2\pi f_1 y_1} \cdot e^{-\e_2 2\pi f_2 y_2} \cdot e^{-\e_4 2\pi f_3 y_3} \nonumber \\
&\qquad\qquad\qquad \cdot e^{\e_4 2\pi f_3 x_3} \cdot e^{\e_2 2\pi f_2 x_2} \cdot e^{\e_1 2\pi f_1 x_1} \,\de\y\,\de\f,\quad i=1,\ldots,7. \label{eq:3_6}
\end{align}

Proof of~\eqref{eq:3_5} can be found in~\cite{BlaszczykSnopek2017} and we only need to prove~\eqref{eq:3_6}. We follow the same steps as in the original proof and use the fact (derived by straightforward calculations) that for any imaginary unit $\e_i$, $i=1,\ldots,7$, we have \begin{align}
&\Big(\big((\e_i\cdot e^{-\e_1 2\pi f_1 y_1}) \cdot e^{-\e_2 2\pi f_2 y_2}\big) \cdot e^{-\e_4 2\pi f_3 y_3}\Big) \cdot e^{\e_4 2\pi f_3 x_3} \nonumber \\
&\qquad = \big((\e_i\cdot e^{-\e_1 2\pi f_1 y_1}) \cdot e^{-\e_2 2\pi f_2 y_2}\big) \cdot \big(e^{-\e_4 2\pi f_3 y_3} \cdot e^{\e_4 2\pi f_3 x_3}\big), \label{eq:6_7}\\
&\big((\e_i\cdot e^{-\e_1 2\pi f_1 y_1}) \cdot e^{-\e_2 2\pi f_2 y_2}\big)\cdot  e^{\e_2 2\pi f_2 x_2} = (\e_i\cdot e^{-\e_1 2\pi f_1 y_1})\cdot( e^{-\e_2 2\pi f_2 y_2} \cdot e^{\e_2 2\pi f_2 x_2}) \label{eq:6_8}\\
&(\e_i\cdot e^{-\e_1 2\pi f_1 y_1}) \cdot e^{\e_1 2\pi f_1 x_1} = \e_i\cdot ( e^{-\e_1 2\pi f_1 y_1} \cdot e^{\e_1 2\pi f_1 x_1}). \label{eq:6_9}
\end{align} Then, using~\eqref{eq:6_7}--\eqref{eq:6_9}, Fubini's Theorem and~Theorem~\ref{th:3_2} we have for $i=1,\ldots,7$ \begin{align*}
&\int_{\R^3} \int_{\R^3}  \bigg(\bigg(\Big(\big((u_i(\y)\e_i\cdot e^{-\e_1 2\pi f_1 y_1}) \cdot e^{-\e_2 2\pi f_2 y_2}\big)\cdot e^{-\e_4 2\pi f_3 y_3}\Big) \\
&\qquad\qquad\qquad\qquad\qquad \cdot e^{\e_4 2\pi f_3 x_3}\bigg)\cdot e^{\e_2 2\pi f_2 x_2}\bigg)\cdot e^{\e_1 2\pi f_1 x_1} \,\de\y\de\f \displaybreak[3]\\
&\quad \!\!\overset{\eqref{eq:6_7}}{=} \int_{\R^2} \int_{\R^2} \bigg(\bigg(\big((\e_i\cdot e^{-\e_1 2\pi f_1 y_1}) \cdot e^{-\e_2 2\pi f_2 y_2}\big) \\
&\qquad\qquad\qquad\qquad\qquad \cdot\left( \int_\R \int_\R u_i(\y) \cdot e^{-\e_4 2\pi f_3 y_3} \cdot e^{\e_4 2\pi f_3 x_3} \,\de y_3\de f_3 \right)\bigg)\cdot e^{\e_2 2\pi f_2 x_2}\bigg)\\
&\qquad\qquad\qquad\qquad\qquad \cdot e^{\e_1 2\pi f_1 x_1} \,\de y_1\de y_2\de f_1\de f_2 \displaybreak[3]\\
&\quad\!\!\!\!\overset{\text{Th.~\ref{th:3_2}}}{=} \int_{\R^2} \int_{\R^2} \bigg(\Big(\big((\e_i\cdot e^{-\e_1 2\pi f_1 y_1}) \cdot e^{-\e_2 2\pi f_2 y_2}\big) \cdot u_i(y_1,y_2,x_3) \Big) \\
&\qquad\qquad\qquad\qquad\qquad \cdot e^{\e_2 2\pi f_2 x_2}\bigg)
\cdot e^{\e_1 2\pi f_1 x_1} \,\de y_1\de y_2\de f_1\de f_2 \displaybreak[3]\\
&\quad = \int_{\R^2} \int_{\R^2} u_i(y_1,y_2,x_3) \cdot \Big(\big((\e_i\cdot e^{-\e_1 2\pi f_1 y_1}) \cdot e^{-\e_2 2\pi f_2 y_2}\big) \cdot e^{\e_2 2\pi f_2 x_2}\Big) \\
&\qquad\qquad\qquad\qquad\qquad \cdot e^{\e_1 2\pi f_1 x_1} \,\de y_1\de y_2\de f_1\de f_2 \displaybreak[3]\\
&\quad \!\!\overset{\eqref{eq:6_8}}{=} \int_{\R} \int_{\R} \bigg((\e_i\cdot e^{-\e_1 2\pi f_1 y_1}) \cdot\left( \int_{\R}\int_{\R} u_i(y_1,y_2,x_3) \cdot e^{-\e_2 2\pi f_2 y_2} \cdot e^{\e_2 2\pi f_2 x_2}\, \de y_2\de f_2\right)\bigg) \\
&\qquad\qquad\qquad\qquad\qquad \cdot e^{\e_1 2\pi f_1 x_1} \,\de y_1\de f_1 \displaybreak[3]\\
&\quad\!\!\!\!\overset{\text{Th.~\ref{th:3_2}}}{=} \int_{\R} \int_{\R} \big((\e_i\cdot e^{-\e_1 2\pi f_1 y_1}) \cdot u_i(y_1,x_2,x_3)\big) \cdot e^{\e_1 2\pi f_1 x_1} \,\de y_1\de f_1 \displaybreak[3]\\
&\quad = \int_{\R} \int_{\R} u_i(y_1,x_2,x_3)\cdot (\e_i\cdot e^{-\e_1 2\pi f_1 y_1}) \cdot e^{\e_1 2\pi f_1 x_1} \,\de y_1\de f_1 \displaybreak[3]\\
&\quad \!\!\overset{\eqref{eq:6_9}}{=} \e_i\cdot \int_{\R} \int_{\R} u_i(y_1,x_2,x_3) \cdot e^{-\e_1 2\pi f_1 y_1}\cdot e^{\e_1 2\pi f_1 x_1} \,\de y_1\de f_1 \displaybreak[3]\\
&\quad\!\!\!\!\overset{\text{Th.~\ref{th:3_2}}}{=} \e_i \cdot u_i(x_1,x_2,x_3).
\end{align*} It concludes the proof.
\end{proof}

It is worth noting that the above theorem was independently proved also in a recent article~\cite{Lian2019}, in which the author used other methods. 
\medskip

Before we proceed to discuss the properties of octonion Fourier transforms, we should start with the basic result formulated below. From now on, we will assume that all the functions under consideration have well-defined octonion Fourier transforms. We will use the convention that the OFT of function $u$ is denoted by $U_{\mathrm{OFT}}$ or $\FourierX{OFT}{u}$. Analogously, the classic (complex) Fourier transform of $u$ will be denoted by $\FourierX{CFT}{u}$.
\medskip

\begin{theorem} \label{the:3_3}
Octonion Fourier transform is $\R$-linear operation, i.e. \begin{align} \label{eq:3_11}
\FourierX{OFT}{a\cdot u+b\cdot v} = a\cdot\FourierX{OFT}{u}+b\cdot\FourierX{OFT}{v},\quad a,b\in\R.
\end{align}
\end{theorem}

It should be noted here that, unlike the classical (complex) Fourier transform (and also quaternion Fourier transform), OFT is not linear in general (to be more precise -- it is not $\O$-linear), i.e. property~\eqref{eq:3_11} is not true for any $a,b\in\O$. This is due to the fact that the octonion multiplication is not associative.
\medskip

For many real- or complex-valued functions the form of the classic Fourier transform is well known. To calculate the octonion Fourier transform of such function, we can use the relationship between these transformations instead of using formula~\eqref{eq:oct_def}. In particular, the following theorem holds, which is the generalization of the result of~\cite{Snopek2012}, where it was proved for real-valued functions. This result was originally stated in~\cite{BlaszczykEUSIPCO2018}, here we complete the details of the proof.
\medskip

\begin{theorem} \label{the:3_4a}
Let $u\colon\R^3\to\C$, $U=\FourierX{CFT}{u}$ and $U_{\mathrm{OFT}}=\FourierX{OFT}{u}$. Then \begin{align} \label{eq:3_12}
U_{\mathrm{OFT}}(f_1,f_2,f_3) &= \frac{1}{4}\big(U(f_1,f_2,f_3)\cdot(1-\e_3)+U(f_1,-f_2,f_3)\cdot(1+\e_3)\big)\cdot(1-\e_5) \nonumber\\
&\quad + \frac{1}{4}\big(U(f_1,f_2,-f_3)\cdot(1-\e_3)+U(f_1,-f_2,-f_3)\cdot(1+\e_3)\big)\cdot(1+\e_5) 
\end{align} where octonion multiplication is done from left to right.
\end{theorem} 

\begin{remark}
Equation (39) proved in~\cite{Snopek2012} may look slightly different from \eqref{eq:3_12}, but after straightforward computation and application of the Hermitian symmetry of the Fourier transform of the real-valued functions we get the abovementioned formula.
\end{remark}

\begin{proof} [of Theorem~\ref{the:3_4a}]
We carefully follow and modify steps presented in~\cite{Snopek2012}. From the definition of the classical Fourier transform we get \begin{align*}
U(f_1,f_2,f_3) = \int_{\R^3} u(\x) e^{-\e_1\alpha_1} e^{-\e_1\alpha_2} e^{-\e_1\alpha_3}\,\de\x,
\end{align*} where $\alpha_j=2\pi f_jx_j$, $j=1,2,3$. From the equivalent definition of sine and cosine functions we get \begin{align}
\frac{1}{2}\big(U(f_1,f_2,f_3)+U(f_1,-f_2,f_3)\big) &= \int_{\R^3} u(\x) e^{-\e_1\alpha_1} (\cos\alpha_2) e^{-\e_1\alpha_3}\,\de\x, \label{eq:3_13}\\
\frac{1}{2}\big(U(f_1,f_2,f_3)-U(f_1,-f_2,f_3)\big) &= \int_{\R^3} u(\x) e^{-\e_1\alpha_1} (-\e_1\sin\alpha_2) e^{-\e_1\alpha_3}\,\de\x. \label{eq:3_14}
\end{align} By changing the sign of $f_3$ in~\eqref{eq:3_14} and multiplying (from the left) by $\e_3$ we get\begin{align} \label{eq:3_15}
\frac{1}{2}\big(U(f_1,f_2,-f_3)-U(f_1,-f_2,-f_3)\big)\e_3 &= \int_{\R^3} u(\x) e^{-\e_1\alpha_1} (\e_2\sin\alpha_2) e^{-\e_1\alpha_3}\,\de\x,
\end{align} which follows from the fact that \begin{align*}
\Big(\big((u\cdot e^{-\e_1\alpha_1})\cdot\e_1\big)\cdot e^{\e_1\alpha_3}\Big)\cdot\e_3 = \big((u\cdot e^{-\e_1\alpha_1})\cdot(\e_1\cdot\e_3)\big)\cdot e^{-\e_1\alpha_3}
\end{align*} (from the fact that octonion multiplication is alternative). Subtracting \eqref{eq:3_15} from \eqref{eq:3_13} we then obtain \begin{align*}
&\frac{1}{2}\big(U(f_1,f_2,f_3)+U(f_1,-f_2,f_3)\big)+\frac{1}{2}\big(U(f_1,-f_2,-f_3)-U(f_1,f_2,-f_3)\big)\e_3 \\
&\qquad\qquad\qquad\qquad\qquad\qquad\qquad\qquad\qquad\qquad =\int_{\R^3} u(\x) e^{-\e_1\alpha_1} e^{-\e_2\alpha_2} e^{-\e_1\alpha_3}\,\de\x.
\end{align*} We introduce the following notation: \begin{align} 
&V(f_1,f_2,f_3) \nonumber\\
&\quad = \frac{1}{2}\big(U(f_1,f_2,f_3)+U(f_1,-f_2,f_3)\big)+\frac{1}{2}\big(U(f_1,-f_2,-f_3)-U(f_1,f_2,-f_3)\big)\e_3. \label{eq:3_15a}
\end{align} By following similar steps as before we get  \begin{align}
\frac{1}{2}\big(V(f_1,f_2,f_3)+V(f_1,f_2,-f_3)\big) &= \int_{\R^3} u(\x) e^{-\e_1\alpha_1} e^{-\e_2\alpha_2}(\cos\alpha_3) \,\de\x, \label{eq:3_16}\\
\frac{1}{2}\big(V(f_1,f_2,f_3)-V(f_1,f_2,-f_3)\big) &= \int_{\R^3} u(\x) e^{-\e_1\alpha_1} e^{-\e_2\alpha_2}(-\e_1\sin\alpha_3) \,\de\x. \label{eq:3_17}
\end{align} 
As earlier we change the sign of $f_2$ in~\eqref{eq:3_17} and multiply (from the left) by~$\e_5$ and obtain \begin{align} \label{eq:3_18}
\frac{1}{2}\big(V(f_1,-f_2,f_3)-V(f_1,-f_2,-f_3)\big)\e_5 &= \int_{\R^3} u(\x) e^{-\e_1\alpha_1} e^{-\e_2\alpha_2}(\e_4\sin\alpha_3) \,\de\x
\end{align} (again from the fact that octonion multiplication is alternative). By subtracting equation~\eqref{eq:3_18} from~\eqref{eq:3_16} we get \begin{align} \label{eq:3_19}
&\frac{1}{2}\big(V(f_1,f_2,f_3)+V(f_1,f_2,-f_3)\big)+\frac{1}{2}\big(V(f_1,-f_2,-f_3)-V(f_1,-f_2,f_3)\big)\e_5 \nonumber\\
&\qquad\qquad\qquad\qquad\qquad\qquad\qquad\qquad\qquad\qquad =\int_{\R^3} u(\x) e^{-\e_1\alpha_1} e^{-\e_2\alpha_2} e^{-\e_4\alpha_3}\,\de\x.
\end{align} We conclude the proof by substituting equation~\eqref{eq:3_15a} in~\eqref{eq:3_19} and regrouping all terms.
\end{proof}

In the general case of an octonion-valued function, well-known formulas for the classic Fourier transform can also be used. If we factor out the complex components of the octonion-valued function, we get \begin{align*}
u &= u_0 + u_1\e_1
 + u_2\e_2 + u_3\e_3
 + u_4\e_4 + u_5\e_5
 + u_6\e_6 + u_7\e_7 \\
&= (u_0 + u_1\e_1)
 + (u_2 + u_3\e_1)\e_2
 + (u_4 + u_5\e_1)\e_4
 + (u_6 + u_7\e_1)\e_2\e_4 \\
&=: v_0 + v_1\e_2 + v_2\e_4 + v_3\e_2\e_4
\end{align*} and $v_0,\ldots,v_3$ are complex-valued functions. Using Theorem~\ref{the:3_4a} we can easily calculate the OFTs of those functions, denote them by $V_0,\ldots,V_3$, respectively. By straightforward calculations we obtain the following properties of octonions: \begin{align}
&\big(((o\cdot\e_2)\cdot e^{-\e_1\alpha_1})\cdot e^{-\e_2\alpha_2}\big)\cdot e^{-\e_4\alpha_3} =
\big(((o\cdot e^{\e_1\alpha_1})\cdot e^{-\e_2\alpha_2})\cdot e^{\e_4\alpha_3}\big)\cdot\e_2, \displaybreak[3] \label{eq:oct1a}\\
&\big(((o\cdot\e_4)\cdot e^{-\e_1\alpha_1})\cdot e^{-\e_2\alpha_2}\big)\cdot e^{-\e_4\alpha_3} =
\big(((o\cdot e^{\e_1\alpha_1})\cdot e^{\e_2\alpha_2})\cdot e^{-\e_4\alpha_3}\big)\cdot\e_4, \displaybreak[3] \label{eq:oct1b}\\
&\big(((o\cdot\e_2\cdot\e_4)\cdot e^{-\e_1\alpha_1})\cdot e^{-\e_2\alpha_2}\big)\cdot e^{-\e_4\alpha_3} =
\big(((o\cdot e^{-\e_1\alpha_1})\cdot e^{\e_2\alpha_2})\cdot e^{\e_4\alpha_3}\big)\cdot\e_2\cdot\e_4, \label{eq:oct1c}
\end{align} for any $o\in\O$ and $\alpha_1,\alpha_2,\alpha_3\in\R$. From those calculations, the corollary below immediately follows. 

\begin{corollary} \label{cor:3_4b}
Let $v_0,\ldots,v_3\colon\R^3\to\C$, $V_i = \FourierX{OFT}{v_i}$ and $u=v_0 + v_1\e_2 + v_2\e_4 + v_3\e_2\e_4$, $U_{\mathrm{OFT}} = \FourierX{OFT}{u}$. Then \begin{align}
U_{\mathrm{OFT}}(f_1,f_2,f_3) &=\phantom{+} V_0(\phantom{-}f_1,\phantom{-}f_2,\phantom{-}f_3)\phantom{\cdot\e_4}\; + V_1(-f_1,\phantom{-}f_2,-f_3)\cdot\e_2 \nonumber \\
& \phantom{=}\, + V_2(-f_1,-f_2,\phantom{-}f_3)\cdot\e_4 + V_3(\phantom{-}f_1,-f_2,-f_3)\cdot\e_2\cdot\e_4. \label{eq:cor3_4b}
\end{align}
\end{corollary}

Inverse octonion Fourier transformation can also be done using classic tools. However, the situation is more complicated from the beginning. In general, the OFT of any function is a function with octonion values. However, we will start, as in the case of forward transform, from the case when a function has OFT with complex values, but in the specific subfield of the octonion algebra, i.e. \begin{align*}
\C_{\e_4} = \{x_0 + x_4\e_4\in\O\colon x_0,x_4\in\R\}.
\end{align*} It is enough to note that in every subfield of this type we can define the classic Fourier transform. In the above case, we have formulas for forward and inverse transforms: \begin{align}
U(f_1,f_2,f_3) &= \int_{\R^3} u(x_1,x_2,x_3) e^{-\e_4\alpha_1} e^{-\e_4\alpha_2} e^{-\e_4\alpha_3}\,\de\x, \nonumber \\
u(x_1,x_2,x_3) &= \int_{\R^3} U(f_1,f_2,f_3) e^{\e_4\alpha_1} e^{\e_4\alpha_2} e^{\e_4\alpha_3}\,\de\f \label{eq:3_19a}
\end{align} where $u\colon\R^3\to\C_{\e_4}$, $\alpha_j=2\pi f_jx_j$, $j=1,2,3$. We have then, of course $U\colon\R^3\to\C_{\e_4}$.

\begin{theorem} \label{the:4_4a}
Let $u\colon\R^3\to\O$ be such that $U_{\mathrm{OFT}}=\FourierX{OFT}{u}\colon \R^3\to\C_{\e_4}$. Moreover, let \linebreak $\hat{u}=\IFourierX{CFT}{U_{\mathrm{OFT}}}$ (in the $\C_{\e_4}$ complex subfield of $\O$, i.e.~\eqref{eq:3_19a}). Then \begin{align} \label{eq:4_1}
u(x_1,x_2,x_3) &= \frac{1}{4}\big(\hat{u}(x_1,x_2,x_3)\cdot(1+\e_6)+\hat{u}(x_1,-x_2,x_3)\cdot(1-\e_6)\big)\cdot(1+\e_5) \nonumber\\
&\quad + \frac{1}{4}\big(\hat{u}(-x_1,x_2,x_3)\cdot(1+\e_6)+\hat{u}(-x_1,-x_2,x_3)\cdot(1-\e_6)\big)\cdot(1-\e_5) 
\end{align} where octonion multiplication is done from left to right.
\end{theorem} 

\begin{proof}
From the modified definition of the classical Fourier transform~\eqref{eq:3_19a} we get \begin{align*}
\hat{u}(x_1,x_2,x_3) = \int_{\R^3} U_{\mathrm{OFT}}(\f) e^{\e_4\alpha_3} e^{\e_4\alpha_2} e^{\e_4\alpha_1}\,\de\f,
\end{align*} where $\alpha_j=2\pi f_jx_j$, $j=1,2,3$. From the equivalent definition of sine and cosine functions we get \begin{align}
\frac{1}{2}\big(\hat{u}(x_1,x_2,x_3)+\hat{u}(x_1,-x_2,x_3)\big) &= \int_{\R^3} U_{\mathrm{OFT}}(\f) e^{\e_4\alpha_3} (\cos\alpha_2) e^{\e_4\alpha_1}\,\de\f, \label{eq:4_13}\\
\frac{1}{2}\big(\hat{u}(x_1,x_2,x_3)-\hat{u}(x_1,-x_2,x_3)\big) &= \int_{\R^3} U_{\mathrm{OFT}}(\f) e^{\e_4\alpha_3} (\e_4\sin\alpha_2) e^{\e_4\alpha_1}\,\de\f. \label{eq:4_14}
\end{align} By changing the sign of $x_1$ in~\eqref{eq:4_14} and multiplying (from the left) by $\e_6$ we get\begin{align} \label{eq:4_15}
\frac{1}{2}\big(\hat{u}(-x_1,x_2,x_3)-\hat{u}(-x_1,-x_2,x_3)\big)\e_6 &= \int_{\R^3} U_{\mathrm{OFT}}(\f) e^{\e_4\alpha_3} (\e_2\sin\alpha_2) e^{\e_4\alpha_1}\,\de\f,
\end{align} which follows from the fact that \begin{align*}
\Big(\big((U_{\mathrm{OFT}}\cdot e^{\e_4\alpha_3})\cdot\e_4\big)\cdot e^{-\e_4\alpha_1}\Big)\cdot\e_6 = \big((U_{\mathrm{OFT}}\cdot e^{\e_4\alpha_3})\cdot(\e_2)\big)\cdot e^{\e_4\alpha_1}
\end{align*} (from the fact that octonion multiplication is alternative). Subtracting \eqref{eq:4_15} from \eqref{eq:4_13} we then obtain \begin{align*}
&\frac{1}{2}\big(\hat{u}(x_1,x_2,x_3)+\hat{u}(x_1,-x_2,x_3)\big)+\frac{1}{2}\big(\hat{u}(-x_1,x_2,x_3)-\hat{u}(-x_1,-x_2,x_3)\big)\e_6 \\
&\qquad\qquad\qquad\qquad\qquad\qquad\qquad\qquad\qquad\qquad =\int_{\R^3} U_{\mathrm{OFT}}(\f) e^{\e_4\alpha_3} e^{\e_2\alpha_2} e^{\e_4\alpha_1}\,\de\f.
\end{align*} We introduce the following notation: \begin{align} 
&w(x_1,x_2,x_3) \nonumber\\
&\quad = \frac{1}{2}\big(\hat{u}(x_1,x_2,x_3)+\hat{u}(x_1,-x_2,x_3)\big)+\frac{1}{2}\big(\hat{u}(-x_1,x_2,x_3)-\hat{u}(-x_1,-x_2,x_3)\big)\e_6. \label{eq:4_15a}
\end{align} By following similar steps as before we get  \begin{align}
\frac{1}{2}\big(w(x_1,x_2,x_3)+w(-x_1,x_2,x_3)\big) &= \int_{\R^3} U_{\mathrm{OFT}}(\f) e^{\e_4\alpha_3} e^{\e_2\alpha_2}(\cos\alpha_1) \,\de\f, \label{eq:4_16}\\
\frac{1}{2}\big(w(x_1,x_2,x_3)-w(-x_1,x_2,x_3)\big) &= \int_{\R^3} U_{\mathrm{OFT}}(\f) e^{\e_4\alpha_3} e^{\e_2\alpha_2}(\e_4\sin\alpha_1) \,\de\f. \label{eq:4_17}
\end{align} 
As earlier we change the sign of $x_2$ in~\eqref{eq:4_17} and multiply (from the left) by~$\e_5$ and obtain \begin{align} \label{eq:4_18}
\frac{1}{2}\big(w(x_1,-x_2,x_3)-w(-x_1,-x_2,x_3)\big)\e_5 &= \int_{\R^3} U_{\mathrm{OFT}}(\f) e^{\e_4\alpha_3} e^{\e_2\alpha_2}(\e_1\sin\alpha_1) \,\de\f
\end{align} (again from the fact that octonion multiplication is alternative). By subtracting equation~\eqref{eq:4_18} from~\eqref{eq:4_16} we get \begin{align} \label{eq:4_19}
&\frac{1}{2}\big(w(x_1,x_2,x_3)+w(-x_1,x_2,x_3)\big)+\frac{1}{2}\big(w(x_1,-x_2,x_3))-w(-x_1,-x_2,x_3)\big)\e_5 \nonumber\\
&\qquad\qquad\qquad\qquad\qquad\qquad\qquad\qquad\qquad\qquad =\int_{\R^3} U_{\mathrm{OFT}}(\f) e^{\e_4\alpha_3} e^{\e_2\alpha_2} e^{\e_1\alpha_1}\,\de\f.
\end{align} We conclude the proof by substituting equation~\eqref{eq:4_15a} in~\eqref{eq:4_19} and regrouping all terms.
\end{proof}

We can now return to the general case. Similarly as before, we factor out the complex ($\C_{\e_4}$) components of the octonion-valued function: \begin{align*}
U_{\mathrm{OFT}} &= U_0 + U_1\e_1 + U_2\e_2 + U_3\e_3 + U_4\e_4 + U_5\e_5 + U_6\e_6 + U_7\e_7 \\
&= (U_0 + U_4\e_4) + (U_1 - U_5\e_4)\e_1 + (U_2 - U_6\e_4)\e_2 + (U_3+U_7\e_4)\e_1\e_2 \\
&:= V_0 + V_1\e_1 + V_2\e_2 + V_3\e_1\e_2
\end{align*} and $V_0,\ldots,V_3$ are functions with values in $\C_{\e_4}$. Using Theorem~\ref{the:4_4a} we can calculate the inverse OFTs of those functions, denote them by $v_0,\ldots,v_3$, respectively. We also use the fact that \begin{align*}
&\big(((z\cdot\e_1)\cdot e^{\e_4\alpha_3})\cdot e^{\e_2\alpha_2}\big)\cdot e^{\e_1\alpha_1} =
\big(((z\cdot e^{-\e_4\alpha_3})\cdot e^{-\e_2\alpha_2})\cdot e^{\e_1\alpha_1}\big)\cdot\e_1, \displaybreak[3]\\
&\big(((z\cdot\e_2)\cdot e^{\e_4\alpha_3})\cdot e^{\e_2\alpha_2}\big)\cdot e^{\e_1\alpha_1} =
\big(((z\cdot e^{-\e_4\alpha_3})\cdot e^{\e_2\alpha_2})\cdot e^{-\e_1\alpha_1}\big)\cdot\e_2, \displaybreak[3]\\
&\big(((z\cdot\e_1\cdot\e_2)\cdot e^{\e_4\alpha_3})\cdot e^{\e_2\alpha_2}\big)\cdot e^{\e_1\alpha_1} =
\big(((z\cdot e^{\e_4\alpha_3})\cdot e^{-\e_2\alpha_2})\cdot e^{-\e_1\alpha_1}\big)\cdot\e_1\cdot\e_2,
\end{align*} for every $z\in\C_{\e_4}$. From those calculation we get the corollary below.

\begin{corollary}
Let $V_0,\ldots,V_3\colon\R^3\to\C_{\e_4}$, $v_i = \IFourierX{OFT}{V_i}$, $i=0,\ldots,3$, and let \linebreak $U_{\mathrm{OFT}}=V_0 + V_1\e_1 + V_2\e_2 + V_3\e_1\e_2$, $u = \IFourierX{OFT}{U_{\mathrm{OFT}}}$. Then \begin{align*}
u(x_1,x_2,x_3) &=\phantom{+} v_0(\phantom{-}x_1,\phantom{-}x_2,\phantom{-}x_3)\phantom{\cdot\e_4}\; + v_1(\phantom{-}x_1,-x_2,-x_3)\cdot\e_1 \\
& \phantom{=}\, + v_2(-x_1,\phantom{-}x_2,-x_3)\cdot\e_2 + v_3(-x_1,-x_2,\phantom{-}x_3)\cdot\e_1\cdot\e_2.
\end{align*}
\end{corollary}

From Theorem~\ref{the:3_4a} and Corollary~\ref{cor:3_4b} one can draw several direct conclusions related to the behavior of the octonion transformation in infinity and the composition of the transformations.

\begin{theorem}
Let $u\colon\R^3\to\O$ be integrable (in Lebesgue sense) and $U_{\mathrm{OFT}} = \FourierX{OFT}{u}$. Then \begin{align*}
\lim\limits_{\abs{\f}\to\infty} U_{\mathrm{OFT}}(\f) = 0.
\end{align*}
\end{theorem}

\begin{proof}
It is a~direct corollary from the classical Riemann-Lebesgue theorem~\cite{Duoandikoetxea2001}, Theorem~\ref{the:3_4a} and Corollary~\ref{cor:3_4b}.
\end{proof}

Before we formulate and prove the next result, let us recall the classical theorem known from Fourier analysis~\cite{Duoandikoetxea2001}.

\begin{theorem}
Let $u\colon\R^n\to\C$ be smooth and rapidly decreasing (i.e. element of Schwartz class). Then \begin{align*}
\FourierX{CFT}{\FourierX{CFT}{u}}(\x) = u(-\x),
\end{align*} and so the classical Fourier transform has period $4$ (i.e. if we apply it four times, we get the identity operator).
\end{theorem}

In the case of OFT, the analogous result is very similar, but slightly more complicated.

\begin{theorem}
Let $u\colon\R^3\to\C$ be smooth and rapidly decreasing (i.e. element of Schwartz class). Then \begin{align*}
\FourierX{OFT}{\FourierX{OFT}{u}}(x_1,x_2,x_3) = \frac{1}{2}\big( &u(x_1,x_2,x_3) + u(-x_1,x_2,-x_3) \\
& \qquad + u(-x_1,-x_2,x_3) - u(x_1,-x_2,-x_3)\big),
\end{align*} and so the OFT has period $4$.
\end{theorem}

\begin{proof}
We begin with the case of $u\colon\R^3\to\C$. Let $U=\FourierX{CFT}{u}$. By carrying out direct calculations, we can write the claim of Theorem~\ref{the:3_4a} in the form \begin{align*}
\FourierX{OFT}{u}(\f) = U_{ee}(\f) - U_{oe}(\f)\cdot\e_1\cdot\e_2 - U_{eo}(\f)\cdot\e_1\cdot\e_4 - U_{oo}(\f)\cdot\e_2\cdot\e_4,
\end{align*} where $U_{yz}$, $y,z\in\{e,o\}$, are four components of $U$ of different parity with respect to $f_2$ and $f_3$, i.e. \begin{align*}
U_{yz}(f_1,f_2,f_3) 
&=\;\; ( \;\;\;\,\;\;\;\,U(\phantom{-}f_1,\phantom{-}f_2,\phantom{-}f_3) + \;\;\;\,\;\;\;\varepsilon_y U(\phantom{-}f_1,-f_2,\phantom{-}f_3) \nonumber \\
&\quad+ \;\;\;\varepsilon_z U(\phantom{-}f_1,\phantom{-}f_2,-f_3) + \;\;\;\varepsilon_y\varepsilon_z U(\phantom{-}f_1,-f_2,-f_3) ) / 4,
\end{align*} where $\varepsilon_y = 1$ if $y=e$ and $\varepsilon_y = -1$ if $y=o$, etc.
\medskip

Note that this is the form as in the assumptions of Corollary~\ref{cor:3_4b}, so when calculating the OFT of function $U_{\mathrm{OFT}}$ we get \begin{align*}
\FourierX{OFT}{\FourierX{OFT}{u}}(x_1,x_2,x_3) &=
\phantom{+} \FourierX{OFT}{\phantom{-}U_{ee}\phantom{\cdot\e_1}\;}(\phantom{-}x_1,\phantom{-}x_2,\phantom{-}x_3) \\
& \phantom{=}\,+ \FourierX{OFT}{- U_{oe}\cdot\e_1}(-x_1,\phantom{-}x_2,-x_3)\cdot\e_2 \\
& \phantom{=}\, + \FourierX{OFT}{- U_{eo}\cdot\e_1}(-x_1,-x_2,\phantom{-}x_3)\cdot\e_4 \\
& \phantom{=}\,+ \FourierX{OFT}{- U_{oo}\phantom{\cdot\e_1}\;}(\phantom{-}x_1,-x_2,-x_3)\cdot\e_2\cdot\e_4.
\end{align*} All functions which OFTs we want to calculate are $\C$-valued functions, so we can again use Theorem~\ref{the:3_4a}. After tedious calculations we get that \begin{align*}
\FourierX{OFT}{\phantom{-}U_{ee}\phantom{\cdot\e_1}\;}(\phantom{-}x_1,\phantom{-}x_2,\phantom{-}x_3)\phantom{\cdot\e_2\cdot\e_4}\; &= \phantom{-}u_{ee}(-x_1,-x_2,-x_3), \\
\FourierX{OFT}{- U_{oe}\cdot\e_1}(-x_1,\phantom{-}x_2,-x_3)\cdot\e_2\phantom{\cdot\e_4}\; &= -u_{oe}(\phantom{-}x_1,-x_2,-x_3), \\
\FourierX{OFT}{- U_{eo}\cdot\e_1}(-x_1,-x_2,\phantom{-}x_3)\cdot\e_4\phantom{\cdot\e_2}\; &= -u_{eo}(\phantom{-}x_1,-x_2,-x_3), \\
\FourierX{OFT}{- U_{oo}\phantom{\cdot\e_1}\;}(\phantom{-}x_1,-x_2,-x_3)\cdot\e_2\cdot\e_4 &= - u_{oo}(-x_1,-x_2,-x_3),
\end{align*} where $u_{yz}$, $y,z\in\{e,o\}$, are four components of $u$ of different parity with respect to $x_2$ and $x_3$, just like earlier. By expanding the above functions and rearranging the components we will receive a claim in the case of functions with complex values. It is easy to calculate that by applying this "double OFT" transformation twice, we get identity.

In the general case of $u\colon\R^3\to\O$ we proceed like in the proof of Corollary~\ref{cor:3_4b}. Let $u = v_0+v_1\e_2+v_2\e_4+v_3\e2\e_4$, where $v_0,\ldots,v_3\colon\R^3\to\C$. Then, from~\eqref{eq:cor3_4b} and~\eqref{eq:oct1a}--\eqref{eq:oct1c} we get \begin{align*}
\FourierX{OFT}{\FourierX{OFT}{u}} &= 
\FourierX{OFT}{\FourierX{OFT}{v_0}} + 
\FourierX{OFT}{\FourierX{OFT}{v_1}}\e_2  \\ &\qquad +
\FourierX{OFT}{\FourierX{OFT}{v_2}}\e_4 +
\FourierX{OFT}{\FourierX{OFT}{v_3}}\e_2\e_4. 
\end{align*} We immediately receive a claim of the theorem from the previous part of the proof.
\end{proof}

In~\cite{HahnSnopek2011a} it was proved that the octonion Fourier transform of the real-valued function can be represented as the octonion sum of components of different parity, i.e. \begin{equation} \label{eq:3_20}
U_{\mathrm{OFT}} = U_{eee} - U_{oee}\e_1 - U_{eoe}\e_2 + U_{ooe}\e_3 - U_{eeo}\e_4 + U_{oeo}\e_5 + U_{eoo}\e_6 - U_{ooo}\e_7,
\end{equation} where $U_{xyz}$, $x,y,z\in\{e,o\}$ are defined as \begin{align}
U_{eee}(\f) &= \int_{\R^3} u_{eee}(\x)\cos(2\pi f_1 x_1)\cos(2\pi f_2 x_2)\cos(2\pi f_3 x_3)\,\de \x, \label{eq:3_20a} \displaybreak[3]\\
U_{oee}(\f) &= \int_{\R^3} u_{oee}(\x)\sin(2\pi f_1 x_1)\cos(2\pi f_2 x_2)\cos(2\pi f_3 x_3)\,\de \x, \displaybreak[3]\\
U_{eoe}(\f) &= \int_{\R^3} u_{eoe}(\x)\cos(2\pi f_1 x_1)\sin(2\pi f_2 x_2)\cos(2\pi f_3 x_3)\,\de \x, \displaybreak[3]\\
U_{ooe}(\f) &= \int_{\R^3} u_{ooe}(\x)\sin(2\pi f_1 x_1)\sin(2\pi f_2 x_2)\cos(2\pi f_3 x_3)\,\de \x, \displaybreak[3]\\
U_{eeo}(\f) &= \int_{\R^3} u_{eeo}(\x)\cos(2\pi f_1 x_1)\cos(2\pi f_2 x_2)\sin(2\pi f_3 x_3)\,\de \x, \displaybreak[3]\\
U_{oeo}(\f) &= \int_{\R^3} u_{oeo}(\x)\sin(2\pi f_1 x_1)\cos(2\pi f_2 x_2)\sin(2\pi f_3 x_3)\,\de \x, \displaybreak[3]\\
U_{eoo}(\f) &= \int_{\R^3} u_{eoo}(\x)\cos(2\pi f_1 x_1)\sin(2\pi f_2 x_2)\sin(2\pi f_3 x_3)\,\de \x, \displaybreak[3]\\
U_{ooo}(\f) &= \int_{\R^3} u_{ooo}(\x)\sin(2\pi f_1 x_1)\sin(2\pi f_2 x_2)\sin(2\pi f_3 x_3)\,\de \x, \label{eq:3_20h} 
\end{align} where $\f=(f_1,f_2,f_3)$, $\x=(x_1,x_2,x_3)$, and functions $u_{xyz}(\x)$, $x,y,z\in\{e,o\}$, are eight components of $u$ of different parity with respect to $x_1$, $x_2$ and~$x_3$, i.e. \begin{align*}
u_{xyz}(x_1,x_2,x_3) 
&=\;\; ( \;\;\;\,\;\;\;\,u(\phantom{-}x_1,\phantom{-}x_2,\phantom{-}x_3) + \;\;\;\,\;\;\;\varepsilon_x u(-x_1,\phantom{-}x_2,\phantom{-}x_3) \nonumber \\
&\quad+ \;\;\;\varepsilon_y u(\phantom{-}x_1,-x_2,\phantom{-}x_3) + \;\;\;\varepsilon_x\varepsilon_y u(-x_1,-x_2,\phantom{-}x_3) \nonumber \\
&\quad+ \;\;\;\varepsilon_z u(\phantom{-}x_1,\phantom{-}x_2,-x_3) + \;\;\;\varepsilon_x\varepsilon_z u(-x_1,\phantom{-}x_2,-x_3) \nonumber \\
&\quad+ \varepsilon_y\varepsilon_z u(\phantom{-}x_1,-x_2,-x_3) + \varepsilon_x\varepsilon_y\varepsilon_z u(-x_1,-x_2,-x_3) ) / 16, 
\end{align*} where $\varepsilon_x = 1$ if $x=e$ and $\varepsilon_x = -1$ if $x=o$, etc.
\medskip

In this notation we use indices $e$ and $o$ to denote that the function is even ($e$) or odd ($o$) with respect to the proper variable, e.g. function $u_{eeo}(\x)$ is even with respect to $x_1$ and~$x_2$ and odd with respect to $x_3$. Analogous considerations can be performed for $U_{xyz}$, e.g. $U_{eoo}$ is even with respect to $f_1$ and odd with respect to $f_2$ and~$f_3$.
\medskip

It should be noted that in case of the real-valued functions $u$, all terms $U_{xyz}$ in \eqref{eq:3_20} are real-valued functions. Similar formulas can be obtained for the octonion-valued functions but we omit the details here.

\section{Properties of the octonion Fourier transform}
\label{sec:3}

\renewcommand{\arraystretch}{1.8}
\begin{table}[!ht]
{\footnotesize\begin{center}
\begin{tabular}{|c|l|l|}
\hline
& function & octonion Fourier transform \\
\hline
1. & $u(\frac{x_1}{a},\frac{x_2}{b},\frac{x_3}{c})$ &
$\abs{abc}U(af_1,bf_2,cf_3)$ \\
\hline
2. & $u(\x)\cdot\cos(2\pi f_0 x_1)$ &
$\big(U(f_1+f_0,f_2,f_3)+U(f_1-f_0,f_2,f_3)\big)\cdot\frac{1}{2}$ \\
   & $u(\x)\cdot\cos(2\pi f_0 x_2)$ &
$\big(U(f_1,f_2+f_0,f_3)+U(f_1,f_2-f_0,f_3)\big)\cdot\frac{1}{2}$ \\
   & $u(\x)\cdot\cos(2\pi f_0 x_3)$ & 
$\big(U(f_1,f_2,f_3+f_0)+U(f_1,f_2,f_3-f_0)\big)\cdot\frac{1}{2}$ \\
\hline
3. & $u(\x)\cdot\sin(2\pi f_0 x_1)$ & 
$\big(U(f_1+f_0,-f_2,-f_3)-U(f_1-f_0,-f_2,-f_3)\big)\cdot\frac{\e_1}{2}$ \\
   & $u(\x)\cdot\sin(2\pi f_0 x_2)$ & 
$\big(U(f_1,f_2+f_0,-f_3)-U(f_1,f_2-f_0,-f_3)\big)\cdot\frac{\e_2}{2}$ \\
   & $u(\x)\cdot\sin(2\pi f_0 x_3)$ & 
$\big(U(f_1,f_2,f_3+f_0)-U(f_1,f_2,f_3-f_0)\big)\cdot\frac{\e_4}{2}$ \\
\hline
4. & $u(\x)\cdot\exp(-\e_1 2\pi f_0 x_1)$ & 
$U(f_1+f_0,f_2,f_3)$ \\
   & $u(\x)\cdot\exp(-\e_2 2\pi f_0 x_2)$ & 
$\big(U(f_1,f_2+f_0,f_3)+U(f_1,f_2-f_0,f_3)$ \\
   & & 
$\qquad +U(-f_1,f_2+f_0,f_3)-U(-f_1,f_2-f_0,f_3)\big)\cdot\frac{1}{2}$ \\
   & $u(\x)\cdot\exp(-\e_4 2\pi f_0 x_3)$ & 
$\big(U(f_1,f_2,f_3+f_0)+U(f_1,f_2,f_3-f_0)$ \\
   & & 
$\qquad +U(-f_1,-f_2,f_3+f_0)-U(-f_1,-f_2,f_3-f_0)\big)\cdot\frac{1}{2}$ \\
\hline
5. & $u(x_1-\alpha,x_2,x_3)$ & 
$\cos(2\pi f_1\alpha) U(f_1,f_2,f_3) - \sin(2\pi f_1\alpha)U(f_1,-f_2,-f_3)\cdot \e_1$ \\
   & $u(x_1,x_2-\beta,x_3)$ & 
$\cos(2\pi f_2\beta) U(f_1,f_2,f_3) - \sin(2\pi f_2\beta)U(f_1,f_2,-f_3)\cdot \e_2$ \\
   & $u(x_1,x_2,x_3-\gamma)$ & 
$\cos(2\pi f_3\gamma) U(f_1,f_2,f_3) - \sin(2\pi f_3\gamma)U(f_1,f_2,f_3)\cdot \e_4$ \\
\hline
6. & $u_{x_1}(\x)$ &
$U(f_1,-f_2,-f_3)\cdot(2\pi f_1\e_1)$ \\
   & $u_{x_2}(\x)$ &
$U(f_1,f_2,-f_3)\cdot(2\pi f_2\e_2)$ \\
   & $u_{x_3}(\x)$ &
$U(f_1,f_2,f_3)\cdot(2\pi f_3\e_4)$ \\
\hline
7. & $2\pi x_1 \cdot u(\x)$ & 
$U_{f_1}(f_1,-f_2,-f_3)\cdot\e_1$ \\
   & $2\pi x_2 \cdot u(\x)$ & 
$U_{f_2}(f_1,f_2,-f_3)\cdot\e_2$ \\
   & $2\pi x_3 \cdot u(\x)$ & 
$U_{f_3}(f_1,f_2,f_3)\cdot\e_4$ \\
\hline
8. & $(u*v)(\x)$ &
$\;\phantom{+}V(\phantom{-}f_1,\phantom{-}f_2,\phantom{-}f_3) \cdot (\phantom{-}U_{eee}(\f)\phantom{\,\e_1} - U_{eeo}(\f)\,\e_4)$ \\
   & & 
$\begin{array}{l} 
+ V(\phantom{-}f_1,-f_2,-f_3) \cdot (-U_{oee}(\f)\,\e_1 + U_{ooe}(\f)\,\e_3) \nonumber\\
+ V(\phantom{-}f_1,\phantom{-}f_2,-f_3) \cdot (-U_{eoe}(\f)\,\e_2 + U_{oeo}(\f)\,\e_5) \\
+ V(-f_1,\phantom{-}f_2,-f_3) \cdot (\phantom{-}U_{eoo}(\f)\,\e_6 - U_{ooo}(\f)\,\e_7) \end{array} $ \\   
\hline
\end{tabular}\end{center}}
\vspace{2mm}
\caption{Summary of octonion Fourier transform properties.}
\label{tab:3_1}
\end{table}

Properties of the complex Fourier transform and its quaternion counterpart are well known in literature~\cite{Allen2004,bulow,Duoandikoetxea2001,Ell2014}. In~\cite{BlaszczykSnopek2017} we already proved some of their octonion analogues (i.e. shift theorem stated in~Theorem~\ref{the:6_8a} and Hermitian symmetry analogue) and in this section we will derive equivalents of other classical properties such as argument scaling and modulation theorem. Sketches of some proofs can be found in~\cite{BlaszczykEUSIPCO2018}, here we will present all previously omitted details, but also present previously unpublished results. The summary of the content of this section is shown in Table~\ref{tab:3_1}.
\medskip

If we do not state otherwise, then in each of the following statements we assume that $u\colon\R^3\to\O$ and~$U=\FourierX{OFT}{u}$.

\begin{theorem} \label{the:3_6}
Let $a,b,c\in\R\setminus\{0\}$ and $v\colon\R^3\to\O$ be defined by $v(x_1,x_2,x_3)=u(\frac{x_1}{a},\frac{x_2}{b},\frac{x_3}{c})$, $V=\FourierX{OFT}{v}$. Then \begin{equation*}
V(f_1,f_2,f_3) = \abs{abc}U(af_1,bf_2,cf_3).
\end{equation*}
\end{theorem}

\begin{proof}
Proof is very similar to the classical case and utilizes integration by substitution. From the definition of the OFT we have \begin{align*}
V(\f) &= \int_{\R^3} u\left(\frac{x_1}{a},\frac{x_2}{b},\frac{x_3}{c}\right)e^{-\e_1 2\pi f_1x_1}e^{-\e_2 2\pi f_2x_2}e^{-\e_4 2\pi f_3x_3}\,\de\x = (\star).
\intertext{We introduce the substitution $(y_1,y_2,y_3)=\left(\frac{x_1}{a},\frac{x_2}{b},\frac{x_3}{c}\right)$. Let us note that determinant of the Jacobian matrix of this substitution is equal to $\abs{\frac{\partial(x_1,x_2,x_3)}{\partial(y_1,y_2,y_3)}}=\abs{abc}$. Then }
(\star)&= \int_{\R^3} u(y_1,y_2,y_3)e^{-\e_1 2\pi af_1y_1}e^{-\e_2 2\pi bf_2y_2}e^{-\e_4 2\pi cf_3y_3}\abs{abc}\,\de\y \\
&=\abs{abc}U(af_1,bf_2,cf_3),
\end{align*} which concludes the proof.
\end{proof}

Theorem~\ref{the:3_6} can be generalized to all linear maps of $\x$. In the case of quaternion Fourier transform one can find similar result in~\cite{bulow} for functions $u\colon\R^2\to\R$ and~$v(\x)=u(\mathbf{A}\x)$, where $\mathbf{A}$ real-valued $2\times 2$ matrix such that $\det(\mathbf{A})\neq0$. Then \begin{align*}
V(f_1,f_2)& = \frac{1}{2\det{\mathbf{A}}} \big(
U(\tilde{a}_{22}f_1+\tilde{a}_{21}f_2,\tilde{a}_{12}f_1+\tilde{a}_{11}f_2) + 
U(\tilde{a}_{22}f_1-\tilde{a}_{21}f_2,-\tilde{a}_{12}f_1+\tilde{a}_{11}f_2) \\
&\!-
\e_3 U(-\tilde{a}_{22}f_1+\tilde{a}_{21}f_2,-\tilde{a}_{12}f_1+\tilde{a}_{11}f_2) +
\e_3U(-\tilde{a}_{22}f_1-\tilde{a}_{21}f_2,\tilde{a}_{12}f_1+\tilde{a}_{11}f_2)\big),
\end{align*} where $\left(\begin{smallmatrix}\tilde{a}_{11} & \tilde{a}_{12} \\ \tilde{a}_{21} & \tilde{a}_{22}\end{smallmatrix}\right)=\frac{1}{\det{\mathbf{A}}}\mathbf{A}$. 
\medskip

In the octonion setup, considering $v(\x)=u(\mathbf{A}\x)$, where $\mathbf{A}$ is some arbitrary nonsingular $3\times 3$ matrix, we would get a result containing 64 different terms. Due to the complication of calculations and slight significance for further research we skip this formula.
\medskip

The next three theorems are known in signal and system theory as the modulation theorem. One can notice that the claim of cosine modulation theorem (with cosine function as a carrier) is exactly the same as in the case of complex Fourier transform. This can not be said about the sine modulation theorem.
\medskip

\begin{theorem} \label{the:6_7}
Let $f_0\in\R$ and denote $u^{\mathrm{cos},i}(\x)=u(\x)\cdot\cos(2\pi f_0x_i)$, $U^{\mathrm{cos},i}=\FourierX{OFT}{u^{\mathrm{cos},i}}$, $i=1,2,3$. Then \begin{align*}
U^{\mathrm{cos},1}(f_1,f_2,f_3) &= \big(U(f_1+f_0,f_2,f_3)+U(f_1-f_0,f_2,f_3)\big)\cdot\frac{1}{2}, \\
U^{\mathrm{cos},2}(f_1,f_2,f_3) &= \big(U(f_1,f_2+f_0,f_3)+U(f_1,f_2-f_0,f_3)\big)\cdot\frac{1}{2}, \\
U^{\mathrm{cos},3}(f_1,f_2,f_3) &= \big(U(f_1,f_2,f_3+f_0)+U(f_1,f_2,f_3-f_0)\big)\cdot\frac{1}{2}.
\end{align*}
\end{theorem}

\begin{proof}
We will use the equivalent definition of the cosine function, i.e.~equation~\eqref{eq:2_4}. Then \begin{align} \label{eq:6_27cos}
\cos\alpha = \frac{1}{2}\big(e^{\e_1 \alpha}+e^{-\e_1 \alpha}\big)
= \frac{1}{2}\big(e^{\e_2 \alpha}+e^{-\e_2 \alpha}\big)
= \frac{1}{2}\big(e^{\e_4 \alpha}+e^{-\e_4 \alpha}\big).
\end{align} Then for $i=1$ we have \begin{align*}
U^{\mathrm{cos},1}(f_1,f_2,f_3) &= \int_{\R^3} \big(u(\x)\cdot\cos(2\pi f_0x_1)\big) e^{-\e_1 2\pi f_1x_1}e^{-\e_2 2\pi f_2x_2}e^{-\e_4 2\pi f_3x_3}\,\de\x \\
&= \int_{\R^3} u(\x)\big(e^{-\e_1 2\pi f_1x_1}\cos(2\pi f_0x_1)\big)e^{-\e_2 2\pi f_2x_2}e^{-\e_4 2\pi f_3x_3}\,\de\x \\
&= \frac{1}{2} \int_{\R^3} u(\x)\big(e^{-\e_1 2\pi f_1x_1}(e^{\e_1 2\pi f_0x_1}+e^{-\e_1 2\pi f_0x_1})\big)e^{-\e_2 2\pi f_2x_2}e^{-\e_4 2\pi f_3x_3}\,\de\x \\
&= \frac{1}{2} \int_{\R^3} u(\x)\big(e^{-\e_1 2\pi (f_1-f_0)x_1}+e^{-\e_1 2\pi (f_1+f_0)x_1}\big)e^{-\e_2 2\pi f_2x_2}e^{-\e_4 2\pi f_3x_3}\,\de\x \\
&= \frac{1}{2}\big(U(f_1-f_0,f_2,f_3)+U(f_1+f_0,f_2,f_3)\big),
\end{align*} which concludes the proof in this case. For $i=2,3$ proceed analogously.
\end{proof}
\medskip

\begin{theorem} \label{the:6_7sin}
Let $f_0\in\R$ and denote $u^{\mathrm{sin},i}(\x)=u(\x)\cdot\sin(2\pi f_0x_i)$, $U^{\mathrm{sin},i}=\FourierX{OFT}{u^{\mathrm{sin},i}}$, \linebreak $i=1,2,3$. Then \begin{align*}
U^{\mathrm{sin},1}(f_1,f_2,f_3) &= \big(U(f_1+f_0,-f_2,-f_3)-U(f_1-f_0,-f_2,-f_3)\big)\cdot\frac{\e_1}{2}, \\
U^{\mathrm{sin},2}(f_1,f_2,f_3) &= \big(U(f_1,f_2+f_0,-f_3)-U(f_1,f_2-f_0,-f_3)\big)\cdot\frac{\e_2}{2}, \\
U^{\mathrm{sin},3}(f_1,f_2,f_3) &= \big(U(f_1,f_2,f_3+f_0)-U(f_1,f_2,f_3-f_0)\big)\cdot\frac{\e_4}{2}.
\end{align*}
\end{theorem}

\begin{proof}
We proceed similarly as in proof of Theorem~\ref{the:6_7}. We will also use the equivalent definition of the sine function formulated in~equation~\eqref{eq:2_4}, i.e. \begin{align} \label{eq:6_27sin}
\sin\alpha = \frac{1}{2\e_1}\big(e^{\e_1 \alpha}-e^{-\e_1 \alpha}\big)
= \frac{1}{2\e_2}\big(e^{\e_2 \alpha}-e^{-\e_2 \alpha}\big)
= \frac{1}{2\e_4}\big(e^{\e_4 \alpha}-e^{-\e_4 \alpha}\big).
\end{align} The following properties of octonion numbers, which can be derived using direct calculations, will also be necessary. For any $o\in\O$ and $\alpha_1,\alpha_2,\alpha_3\in\R$ we have \begin{align}
&\Big(\big(o\cdot (e^{-\e_1\alpha_1}\cdot\e_1)\big)\cdot e^{-\e_2\alpha_2}\Big)\cdot e^{-\e_4\alpha_3} =
\Big(\big((o\cdot e^{-\e_1\alpha_1})\cdot e^{\e_2\alpha_2}\big)\cdot e^{\e_4\alpha_3}\Big)\cdot\e_1, \label{eq:6_28}\\
& \big((o\cdot e^{-\e_1\alpha_1})\cdot (e^{-\e_2\alpha_2}\cdot\e_2)\big)\cdot e^{-\e_4\alpha_3} =
\Big(\big((o\cdot e^{-\e_1\alpha_1})\cdot e^{-\e_2\alpha_2}\big)\cdot e^{\e_4\alpha_3}\Big)\cdot\e_2, \label{eq:6_29}\\
& \big((o\cdot e^{-\e_1\alpha_1})\cdot e^{-\e_2\alpha_2}\big)\cdot (e^{-\e_4\alpha_3}\cdot\e_4) =
\Big(\big((o\cdot e^{-\e_1\alpha_1})\cdot e^{-\e_2\alpha_2}\big)\cdot e^{-\e_4\alpha_3}\Big)\cdot\e_4. \label{eq:6_30}
\end{align} Then, for $i=1$ we have \begin{align*}
U^{\mathrm{sin},1}(f_1,f_2,f_3) &= \int_{\R^3} \big(u(\x)\cdot\sin(2\pi f_0x_1)\big) e^{-\e_1 2\pi f_1x_1}e^{-\e_2 2\pi f_2x_2}e^{-\e_4 2\pi f_3x_3}\,\de\x \displaybreak[3]\\
&= \int_{\R^3} u(\x)\big(e^{-\e_1 2\pi f_1x_1}\sin(2\pi f_0x_1)\big)e^{-\e_2 2\pi f_2x_2}e^{-\e_4 2\pi f_3x_3}\,\de\x \displaybreak[3]\\
&= -\int_{\R^3} u(\x)\big(e^{-\e_1 2\pi f_1x_1}\cdot\e_1\sin(2\pi f_0x_1)\cdot\e_1\big)e^{-\e_2 2\pi f_2x_2}e^{-\e_4 2\pi f_3x_3}\,\de\x \displaybreak[3]\\
&\!\!\!\overset{\eqref{eq:6_28}}{=}\!\!\! -\frac{1}{2} \int_{\R^3} u(\x)\big(e^{-\e_1 2\pi f_1x_1}(e^{\e_1 2\pi f_0x_1}-e^{-\e_1 2\pi f_0x_1})\big)e^{\e_2 2\pi f_2x_2}e^{\e_4 2\pi f_3x_3}\,\de\x \cdot\e_1\\
&= -\frac{1}{2} \int_{\R^3} u(\x)\big(e^{-\e_1 2\pi (f_1-f_0)x_1}-e^{-\e_1 2\pi (f_1+f_0)x_1}\big)e^{-\e_2 2\pi f_2x_2}e^{-\e_4 2\pi f_3x_3}\,\de\x \\
&= -\frac{1}{2}\big(U(f_1-f_0,f_2,f_3)-U(f_1+f_0,f_2,f_3)\big)\cdot\e_1,
\end{align*} which concludes the proof. For $i=2,3$ the property is proved analogously, using equations~\eqref{eq:6_29} and~\eqref{eq:6_30}.
\end{proof}

It may seem that the next theorem is a simple consequence of Theorem~\ref{the:6_7} and~\ref{the:6_7sin}. However, that is not the case since OFT is not a~$\O$-linear operation. We need to prove it independently.

\begin{theorem} \label{the:6_8}
Let $f_0\in\R$ and $u^{\mathrm{exp},i}(\x)=u(\x)\cdot\exp(-\e_{2^{i-1}}2\pi f_0x_i)$, $U^{\mathrm{exp},i}=\FourierX{OFT}{u^{\mathrm{exp},i}}$, $i=1,2,3$. Then \begin{align*}
U^{\mathrm{exp},1}(f_1,f_2,f_3) &= U(f_1+f_0,f_2,f_3), \\
U^{\mathrm{exp},2}(f_1,f_2,f_3) &= \big(U(f_1,f_2+f_0,f_3)+U(f_1,f_2-f_0,f_3) \nonumber\\
&\qquad +U(-f_1,f_2+f_0,f_3)-U(-f_1,f_2-f_0,f_3)\big)\cdot\frac{1}{2}, \\
U^{\mathrm{exp},3}(f_1,f_2,f_3) &= \big(U(f_1,f_2,f_3+f_0)+U(f_1,f_2,f_3-f_0) \nonumber\\
&\qquad +U(-f_1,-f_2,f_3+f_0)-U(-f_1,-f_2,f_3-f_0)\big)\cdot\frac{1}{2}.
\end{align*}
\end{theorem}

\begin{proof}
Properties in the claim of this theorem are proved similarly as those of Theorem~\ref{the:6_7} and~\ref{the:6_7sin} -- using equations~\eqref{eq:6_27cos} and~\eqref{eq:6_27sin}. Furthermore we will use the following fact -- for each $o\in\O$ and $\alpha_1,\alpha_2,\alpha_3\in\R$ we have: \begin{align}
&\Big(\big((o\cdot\e_1)\cdot e^{-\e_1\alpha_1}\big)\cdot e^{-\e_2\alpha_2}\Big)\cdot e^{-\e_4\alpha_3} =
\Big(\big(o\cdot (e^{-\e_1\alpha_1}\cdot\e_1)\big)\cdot e^{-\e_2\alpha_2}\Big)\cdot e^{-\e_4\alpha_3}, \label{eq:6_31}\\
&\Big(\big((o\cdot\e_2)\cdot e^{-\e_1\alpha_1}\big)\cdot e^{-\e_2\alpha_2}\Big)\cdot e^{-\e_4\alpha_3} =
\big((o\cdot e^{\e_1\alpha_1})\cdot (e^{-\e_2\alpha_2}\cdot\e_2)\big)\cdot e^{-\e_4\alpha_3}, \label{eq:6_32}\\
&\Big(\big((o\cdot\e_4)\cdot e^{-\e_1\alpha_1}\big)\cdot e^{-\e_2\alpha_2}\Big)\cdot e^{-\e_4\alpha_3} =
\big((o\cdot e^{\e_1\alpha_1})\cdot e^{\e_2\alpha_2}\big)\cdot (e^{-\e_4\alpha_3}\cdot\e_4). \label{eq:6_33}
\end{align} Then for $i=3$ we have \begin{align*}
U^{\mathrm{exp},3}(f_1&,f_2,f_3) = \int_{\R^3} \big(u(\x)\cdot e^{-\e_4 2\pi f_0x_3}\big) e^{-\e_1 2\pi f_1x_1}e^{-\e_2 2\pi f_2x_2}e^{-\e_4 2\pi f_3x_3}\,\de\x \displaybreak[3]\\
&= \int_{\R^3} \big(u(\x)\cdot (\cos(2\pi f_0x_3)-\e_4\sin(2\pi f_0x_3)) \big) e^{-\e_1 2\pi f_1x_1}e^{-\e_2 2\pi f_2x_2}e^{-\e_4 2\pi f_3x_3}\,\de\x \displaybreak[3]\\
&= \int_{\R^3} \big(u(\x)\cdot \cos(2\pi f_0x_3)\big) e^{-\e_1 2\pi f_1x_1}e^{-\e_2 2\pi f_2x_2}e^{-\e_4 2\pi f_3x_3}\,\de\x \\
&\quad -\int_{\R^3} \big(u(\x)\cdot \e_4\sin(2\pi f_0x_3) \big) e^{-\e_1 2\pi f_1x_1}e^{-\e_2 2\pi f_2x_2}e^{-\e_4 2\pi f_3x_3}\,\de\x \displaybreak[3]\\
&\!\!\!\overset{\eqref{eq:6_33}}{=}\!\!\!  \int_{\R^3} u(\x)e^{-\e_1 2\pi f_1x_1}e^{-\e_2 2\pi f_2x_2}\big(e^{-\e_4 2\pi f_3x_3}\cdot \cos(2\pi f_0x_3)\big) \,\de\x \\
&\quad -\int_{\R^3} u(\x) e^{\e_1 2\pi f_1x_1}e^{\e_2 2\pi f_2x_2}\big(e^{-\e_4 2\pi f_3x_3}\cdot \e_4\sin(2\pi f_0x_3)\big)\,\de\x \displaybreak[3]\\
&= \frac{1}{2}\big(U(f_1,f_2,f_3-f_0)+U(f_1,f_2,f_3+f_0)\big) \\
&\quad -\frac{1}{2}\big(U(-f_1,-f_2,f_3-f_0)-U(-f_1,-f_2,f_3+f_0)\big),
\end{align*} which concludes the proof in this case. Proofs for $i=1,2$ are similar and use equations~\eqref{eq:6_31} and~\eqref{eq:6_32}.
\end{proof}

To complete our considerations about properties of the octonion Fourier transforms of octonion-valued functions, we should also state and prove the shift theorem. In case of real-valued functions we already stated this theorem in our earlier work, i.e. article~\cite{BlaszczykSnopek2017}. 

\begin{theorem} \label{the:6_8a}
Let $\alpha,\beta,\gamma\in\R$ and denote $u^{\alpha}(\x)=u(x_1-\alpha,x_2,x_3)$, $u^{\beta}(\x)=u(x_1,x_2-\beta,x_3)$ and~$u^{\gamma}(\x)=u(x_1,x_2,x_3-\gamma)$. Let $U^\ell=\FourierX{OFT}{u^\ell}$, $\ell=\alpha,\beta,\gamma$. Then \begin{align}
U^{\alpha}(f_1,f_2,f_3) &= \cos(2\pi f_1\alpha) U(f_1,f_2,f_3) - \sin(2\pi f_1\alpha)U(f_1,-f_2,-f_3)\cdot \e_1, \\
U^{\beta}(f_1,f_2,f_3) &= \cos(2\pi f_2\beta) U(f_1,f_2,f_3) - \sin(2\pi f_2\beta)U(f_1,\;\;\; f_2,-f_3)\cdot \e_2, \label{eq:6_34}\\
U^{\gamma}(f_1,f_2,f_3) &= \cos(2\pi f_3\gamma) U(f_1,f_2,f_3) - \sin(2\pi f_3\gamma)U(f_1,\;\;\; f_2,\;\;\; f_3)\cdot \e_4. \label{eq:6_35}
\end{align}
\end{theorem}

\begin{proof}
We will use again the tools used in proofs of previous claims. Consider the OFT of function $u^{\alpha}$. Using integration by substitution we get \begin{align*}
U^{\alpha}(f_1&,f_2,f_3) = \int_{\R^3} u(x_1-\alpha,x_2,x_3) e^{-\e_1 2\pi f_1x_1}e^{-\e_2 2\pi f_2x_2}e^{-\e_4 2\pi f_3x_3}\,\de\x \displaybreak[3]\\
&= \int_{\R^3} u(x_1,x_2,x_3) e^{-\e_1 2\pi f_1(x_1+\alpha)}e^{-\e_2 2\pi f_2x_2}e^{-\e_4 2\pi f_3x_3}\,\de\x \displaybreak[3] \\
&= \int_{\R^3} u(\x) (e^{-\e_1 2\pi f_1x_1}\cdot e^{-\e_1 2\pi f_1\alpha})e^{-\e_2 2\pi f_2x_2}e^{-\e_4 2\pi f_3x_3}\,\de\x \displaybreak[3] \\
&= \int_{\R^3} u(\x) \big(e^{-\e_1 2\pi f_1x_1}\cdot (\cos(2\pi f_1\alpha)-\e_1\sin(2\pi f_1\alpha))\big)e^{-\e_2 2\pi f_2x_2}e^{-\e_4 2\pi f_3x_3}\,\de\x \displaybreak[3] \\
&\!\!\!\overset{\eqref{eq:6_28}}{=}\!\!\!   \cos(2\pi f_1\alpha) \int_{\R^3} u(\x) e^{-\e_1 2\pi f_1x_1}e^{-\e_2 2\pi f_2x_2}e^{-\e_4 2\pi f_3x_3}\,\de\x \\
&\quad-\sin(2\pi f_1\alpha) \int_{\R^3} u(\x) e^{-\e_1 2\pi f_1x_1}e^{\e_2 2\pi f_2x_2}e^{\e_4 2\pi f_3x_3}\,\de\x\cdot\e_1 \displaybreak[3]\\
&=\cos(2\pi f_1\alpha) U(f_1,f_2,f_3) - \sin(2\pi f_1\alpha)U(f_1,-f_2,-f_3)\cdot \e_1,
\end{align*} which concludes the proof for $u^{\alpha}$. The derivation of \eqref{eq:6_34} and~\eqref{eq:6_35} is very similar and utilises properties~\eqref{eq:6_29} and~\eqref{eq:6_30}.
\end{proof}

Next properties that we will prove are a~key element in the analysis of multidimensional linear time-invariant systems described by a~system of partial differential equations. In our considerations, however, we will limit ourselves to real-valued functions and from now on we assume that $u,v\colon\R^3\to\R$ and $U$ and $V$ are the OFTs of $u$ and $v$, respectively. We will also assume that the relevant derivatives of $u$ exist, as well as their OFTs.
\medskip

\begin{theorem} \label{thm:deriv}
Let $U^{\partial x_1}$, $U^{\partial x_2}$ and $U^{\partial x_3}$ denote the OFTs of $u_{x_1}$, $u_{x_2}$ and $u_{x_3}$, respectively. Then \begin{align}
U^{\partial x_1}(f_1,f_2,f_3) &= U(f_1, -f_2, -f_3)\cdot(2\pi f_1\e_1), \label{eq:3_1} \\
U^{\partial x_2}(f_1,f_2,f_3) &= U(f_1, \phantom{-}f_2, -f_3)\cdot(2\pi f_2\e_2), \\
U^{\partial x_3}(f_1,f_2,f_3) &= U(f_1, \phantom{-}f_2, \phantom{-}f_3)\cdot(2\pi f_3\e_4).
\end{align}
\end{theorem}

\begin{proof}
We will prove only the first formula. Consider derivative $u_{x_1}$ and its octonion Fourier transform $U^{\partial x_1}$. $u_{x_1}$ is a~real-valued function, hence we can write $U^{\partial x_1}$ as a~sum of eight components of different parity \begin{align*}
U^{\partial x_1} = U^{\partial x_1}_{eee} - U^{\partial x_1}_{oee}\e_1 - U^{\partial x_1}_{eoe}\e_2 + U^{\partial x_1}_{ooe}\e_3 - U^{\partial x_1}_{eeo}\e_4 + U^{\partial x_1}_{oeo}\e_5 + U^{\partial x_1}_{eoo}\e_6 - U^{\partial x_1}_{ooo}\e_7,
\end{align*} where \begin{align*}
U^{\partial x_1}_{eee}(\f) &= \int_{\R^3} u_{x_1}(\x)\cos(2\pi f_1 x_1)\cos(2\pi f_2 x_2)\cos(2\pi f_3 x_3)\,\de \x, \\
U^{\partial x_1}_{oee}(\f) &= \int_{\R^3} u_{x_1}(\x)\sin(2\pi f_1 x_1)\cos(2\pi f_2 x_2)\cos(2\pi f_3 x_3)\,\de \x, \displaybreak[3]\\
U^{\partial x_1}_{eoe}(\f) &= \int_{\R^3} u_{x_1}(\x)\cos(2\pi f_1 x_1)\sin(2\pi f_2 x_2)\cos(2\pi f_3 x_3)\,\de \x, \displaybreak[3]\\
U^{\partial x_1}_{ooe}(\f) &= \int_{\R^3} u_{x_1}(\x)\sin(2\pi f_1 x_1)\sin(2\pi f_2 x_2)\cos(2\pi f_3 x_3)\,\de \mathbf{x}, \displaybreak[3]\\
U^{\partial x_1}_{eeo}(\f) &= \int_{\R^3} u_{x_1}(\x)\cos(2\pi f_1 x_1)\cos(2\pi f_2 x_2)\sin(2\pi f_3 x_3)\,\de \x, \displaybreak[3]\\
U^{\partial x_1}_{oeo}(\f) &= \int_{\R^3} u_{x_1}(\x)\sin(2\pi f_1 x_1)\cos(2\pi f_2 x_2)\sin(2\pi f_3 x_3)\,\de \x, \displaybreak[3]\\
U^{\partial x_1}_{eoo}(\f) &= \int_{\R^3} u_{x_1}(\x)\cos(2\pi f_1 x_1)\sin(2\pi f_2 x_2)\sin(2\pi f_3 x_3)\,\de \x, \displaybreak[3]\\
U^{\partial x_1}_{ooo}(\f) &= \int_{\R^3} u_{x_1}(\x)\sin(2\pi f_1 x_1)\sin(2\pi f_2 x_2)\sin(2\pi f_3 x_3)\,\de \x, 
\end{align*} where $\f=(f_1,f_2,f_3)$, $\x=(x_1,x_2,x_3)$. 
\medskip

We will use integration by parts and utilize the fact that for every integrable and smooth function $u$ and every $x_2,x_3\in\R$ we have $\lim\limits_{x_1\to\pm\infty} u(\x)=0$. Then \begin{align*}
U^{\partial x_1}_{eee}(\f) &= \phantom{-}\int_{\R^3} u(\x)\sin(2\pi f_1 x_1)\cos(2\pi f_2 x_2)\cos(2\pi f_3 x_3)\,\de \x\cdot 2\pi f_1=\phantom{-}U_{oee}(\f)\cdot 2\pi f_1, \displaybreak[3]\\
U^{\partial x_1}_{oee}(\f) &= -\int_{\R^3} u(\x)\cos(2\pi f_1 x_1)\cos(2\pi f_2 x_2)\cos(2\pi f_3 x_3)\,\de \x\cdot 2\pi f_1=-U_{eee}(\f)\cdot 2\pi f_1, \displaybreak[3]\\
U^{\partial x_1}_{eoe}(\f) &= \phantom{-}\int_{\R^3} u(\x)\sin(2\pi f_1 x_1)\sin(2\pi f_2 x_2)\cos(2\pi f_3 x_3)\,\de \x\cdot 2\pi f_1=\phantom{-}U_{ooe}(\f)\cdot 2\pi f_1, \displaybreak[3]\\
U^{\partial x_1}_{ooe}(\f) &= -\int_{\R^3} u(\x)\cos(2\pi f_1 x_1)\sin(2\pi f_2 x_2)\cos(2\pi f_3 x_3)\,\de \x\cdot 2\pi f_1=-U_{eoe}(\f)\cdot 2\pi f_1, \displaybreak[3]\\
U^{\partial x_1}_{eeo}(\f) &= \phantom{-}\int_{\R^3} u(\x)\sin(2\pi f_1 x_1)\cos(2\pi f_2 x_2)\sin(2\pi f_3 x_3)\,\de \x\cdot 2\pi f_1=\phantom{-}U_{oeo}(\f)\cdot 2\pi f_1, \displaybreak[3]\\
U^{\partial x_1}_{oeo}(\f) &= -\int_{\R^3} u(\x)\cos(2\pi f_1 x_1)\cos(2\pi f_2 x_2)\sin(2\pi f_3 x_3)\,\de \x\cdot 2\pi f_1=-U_{eeo}(\f)\cdot 2\pi f_1, \displaybreak[3]\\
U^{\partial x_1}_{eoo}(\f) &= \phantom{-}\int_{\R^3} u(\x)\sin(2\pi f_1 x_1)\sin(2\pi f_2 x_2)\sin(2\pi f_3 x_3)\,\de \x\cdot 2\pi f_1=\phantom{-}U_{ooo}(\f)\cdot 2\pi f_1, \displaybreak[3]\\
U^{\partial x_1}_{ooo}(\f) &= -\int_{\R^3} u(\x)\cos(2\pi f_1 x_1)\sin(2\pi f_2 x_2)\sin(2\pi f_3 x_3)\,\de \x\cdot 2\pi f_1=-U_{eoo}(\f)\cdot 2\pi f_1, 
\end{align*} where \begin{align*}
U = U_{eee} - U_{oee}\e_1 - U_{eoe}\e_2 + U_{ooe}\e_3 - U_{eeo}\e_4 + U_{oeo}\e_5 + U_{eoo}\e_6 - U_{ooo}\e_7
\end{align*} is a~sum of eight components of different parity, as explained in~Section~\ref{sec:2}. Hence \begin{align*}
U^{\partial x_1} &= \big(U_{oee} + U_{eee}\e_1 - U_{ooe}\e_2 - U_{eoe}\e_3 - U_{oeo}\e_4 - U_{eeo}\e_5 + U_{ooo}\e_6 + U_{eoo}\e_7\big)\cdot 2\pi f_1 \\
&= \big(U_{eee} - U_{oee}\e_1 + U_{eoe}\e_2 - U_{ooe}\e_3 + U_{eeo}\e_4 - U_{oeo}\e_5 + U_{eoo}\e_6 - U_{ooo}\e_7\big)\cdot (2\pi f_1 \e_1).
\end{align*} Considering the parity of each component we get formula~\eqref{eq:3_1}. 
\end{proof}

Let us note that the statement of the above theorem is analogous to the claim of the classic version of the Fourier transform of the derivative. The difference is first of all the kind of imaginary unit by which the Fourier transform is multiplied and the change of sign at some variables. This is a characteristic feature of the octonion Fourier transformation. Similar argument leads to the following corollaries for the OFTs of partial derivatives of higher order. 
\medskip

\begin{corollary} \label{cor:3_1}
Let $U^{\partial x_i \ldots x_j}$ denote the OFT of $u_{x_i \ldots x_j}$. Then \begin{align*}
U^{\partial x_1 x_2}(f_1,f_2,f_3) &= U(\phantom{-}f_1,-f_2,-f_3)\cdot(2\pi f_1)(2\pi f_2) \e_3, \\
U^{\partial x_1 x_3}(f_1,f_2,f_3) &= U(\phantom{-}f_1,\phantom{-}f_2,-f_3)\cdot(2\pi f_1)(2\pi f_3)\e_5, \\
U^{\partial x_2 x_3}(f_1,f_2,f_3) &= U(-f_1,\phantom{-}f_2,-f_3)\cdot(2\pi f_2)(2\pi f_3)\e_6, \\
U^{\partial x_1 x_2 x_3}(f_1,f_2,f_3) &= U(-f_1,\phantom{-}f_2,-f_3)\cdot(2\pi f_1)(2\pi f_2)(2\pi f_3)\e_7.
\end{align*}
\end{corollary}

An analogous conclusion can also be drawn for pure partial derivatives of the second order. It is worth noting that the claim of the theorem is no different from the corresponding theorem for the classic Fourier transform. We leave claims of Corollary~\ref{cor:3_1} and~\ref{cor:3_2} without proof.

\begin{corollary} \label{cor:3_2}
Let $U^{\partial x_i x_i}$ denote the OFT of $u_{x_i x_i}$. Then \begin{align*}
U^{\partial x_1 x_1}(f_1,f_2,f_3) &= -U(f_1,f_2,f_3)\cdot(2\pi f_1)^2,  \\
U^{\partial x_2 x_2}(f_1,f_2,f_3) &= -U(f_1,f_2,f_3)\cdot(2\pi f_2)^2, \\
U^{\partial x_3 x_3}(f_1,f_2,f_3) &= -U(f_1,f_2,f_3)\cdot(2\pi f_3)^2.
\end{align*}
\end{corollary}

Another significant result that can be demonstrated is the OFT differentiation theorem. First of all, however, the concept of differentiation of octonion-valued function should be defined. We will say that the partial derivative $v_{x_i}$, $i=1,2,3$, of the function $v\colon\R^3\to\O$, $v = v_0 + v_1\e_1 + \ldots + v_7\e_7$, exists if and only if all the partial derivatives $v_{j,x_i}$, $i=1,2,3$, $j=0,\ldots,7$, exist and then: \begin{align*}
v_{x_i}(\x) = v_{0,x_i} + v_{1,x_i}(\x)\e_1 + \ldots + v_{7,x_i}(\x)\e_7.
\end{align*} We can now formulate the theorem on partial derivatives of the OFT $U$ of $u\colon\R^3\to\R$, analogous to the theorem known from the classical Fourier analysis. As before, we will assume that all considered derivatives and transforms exist.

\begin{theorem}
Let $V_i$ and $W_i$, where $i=1,2,3$, denote the OFTs of $v_i(\x) = -2\pi x_i u(\x)$ and $w_i(\x) = v_i(\x)\cdot\e_{2^{i-1}}$, $i=1,2,3$, respectively. Then \begin{align}
U_{f_1}(f_1,f_2,f_3) &= V_1(f_1,-f_2,-f_3)\cdot\e_1 = W_1(\phantom{-}f_1,\phantom{-}f_2,f_3), \label{eq:oft_der_x1}\\
U_{f_2}(f_1,f_2,f_3) &= V_2(f_1,\phantom{-}f_2,-f_3)\cdot\e_2 = W_2(-f_1,\phantom{-}f_2,f_3), \\
U_{f_3}(f_1,f_2,f_3) &= V_3(f_1,\phantom{-}f_2,\phantom{-}f_3)\cdot\e_4 = W_3(-f_1,-f_2,f_3).
\end{align}
\end{theorem}

\begin{proof}
We will apply methods analogous to those used in the proof of Theorem~\ref{thm:deriv}. We will prove only the first of the given formulas, the remaining ones are shown in the same way. Consider function $v(\x) = -2\pi x_1u(\x)$ and notice that \begin{align*}
v_{eyz}(\x) = -2\pi x_1 u_{oyz}(\x), \quad y,z\in\{e,o\}, \\
v_{oyz}(\x) = -2\pi x_1 u_{eyz}(\x), \quad y,z\in\{e,o\}.
\end{align*} As in the previous considerations, we can write $U$ and $V$ (the OFT of $v$) as sums of eight components of different parity \begin{align*}
U &= U_{eee} - U_{oee}\e_1 - U_{eoe}\e_2 + U_{ooe}\e_3 - U_{eeo}\e_4 + U_{oeo}\e_5 + U_{eoo}\e_6 - U_{ooo}\e_7, \\
V &= V_{eee} - V_{oee}\e_1 - V_{eoe}\e_2 + V_{ooe}\e_3 - V_{eeo}\e_4 + V_{oeo}\e_5 + V_{eoo}\e_6 - V_{ooo}\e_7,
\end{align*} where $U_{xyz}$, $x,y,z\in\{e,o\}$ are defined as in~\eqref{eq:3_20a}--\eqref{eq:3_20h} and $V_{xyz}$ analogously.

Assuming that the functions $u$ and $v$ are integrable (in the Lebesgue sense), we can differentiate under the integral sign and then we get \begin{align*}
U_{eee,f_1}(\f) &= 
- \int_{\R^3} 2\pi x_1 u_{eee}(\x) \sin(2\pi f_1 x_1) \cos(2\pi f_2 x_2) \cos(2\pi f_3 x_3)\,\de\x = \phantom{-}V_{oee}(\f), \\
U_{oee,f_1}(\f) &= 
\phantom{-} \int_{\R^3} 2\pi x_1 u_{oee}(\x) \cos(2\pi f_1 x_1) \cos(2\pi f_2 x_2) \cos(2\pi f_3 x_3)\,\de\x = -V_{eee}(\f), \\
&\vdots \\
U_{eoo,f_1}(\f) &= 
- \int_{\R^3} 2\pi x_1 u_{eoo}(\x) \sin(2\pi f_1 x_1) \sin(2\pi f_2 x_2) \sin(2\pi f_3 x_3)\,\de\x = \phantom{-}V_{ooo}(\f), \\
U_{ooo,f_1}(\f) &= 
\phantom{-} \int_{\R^3} 2\pi x_1 u_{ooo}(\x) \sin(2\pi f_1 x_1) \sin(2\pi f_2 x_2) \sin(2\pi f_3 x_3)\,\de\x = -V_{eoo}(\f).
\end{align*} Hence \begin{align*}
U_{f_1} &= V_{oee} + V_{eee}\e_1 - V_{ooe}\e_2 - V_{eoe}\e_3 - V_{oeo}\e_4 - V_{eeo}\e_5 + V_{ooo}\e_6 + V_{eoo}\e_7 \\
&= (V_{eee} - V_{oee}\e_1 + V_{eoe}\e_2 - V_{ooe}\e_3 + V_{eeo}\e_4 - V_{oeo}\e_5 + V_{eoo}\e_6 - V_{ooo}\e_7)\cdot\e_1.
\end{align*} Considering the parity of each component we get the first equality in formula~\eqref{eq:oft_der_x1}. Using the fact that for any $o\in\O$ and $\alpha_1,\alpha_2,\alpha_3\in\R$ we have \begin{align*}
\big(((o\cdot\e_1)\cdot e^{-\e_1\alpha_1})\cdot e^{-\e_2\alpha_2}\big)\cdot e^{-\e_4\alpha_3} =
\big(((o\cdot e^{-\e_1\alpha_1})\cdot e^{\e_2\alpha_2})\cdot e^{\e_4\alpha_3}\big)\cdot\e_1
\end{align*} we get the second equality in~\eqref{eq:oft_der_x1}, which concludes the proof.
\end{proof}

At the end of this section we go to one of the most important Fourier transform properties and the most frequently used in signal analysis -- the so-called Convolution theorem. This claim was already signaled in~\cite{BlaszczykECMI2018}, here we present the details of the proof.
\medskip

\begin{theorem} \label{the:3_4}
Let $\mathcal{F}_{\mathrm{OFT}}\{u*v\}$ denote the OFT of the convolution of $u$ and $v$, i.e. \begin{align*}
(u*v)(\x) = \int_{\R^3} u(\y)\cdot v(\x-\y)\,\de\y. 
\end{align*} Then \begin{align*}
\mathcal{F}_{\mathrm{OFT}}\{u*v\}(\f) &=\phantom{+} V(\phantom{-}f_1,\phantom{-}f_2,\phantom{-}f_3) \cdot (\phantom{-}U_{eee}(\f)\phantom{\,\e_1} - U_{eeo}(\f)\,\e_4) \\
&\phantom{=}+ V(\phantom{-}f_1,-f_2,-f_3) \cdot (-U_{oee}(\f)\,\e_1 + U_{ooe}(\f)\,\e_3) \nonumber\\
&\phantom{=}+ V(\phantom{-}f_1,\phantom{-}f_2,-f_3) \cdot (-U_{eoe}(\f)\,\e_2 + U_{oeo}(\f)\,\e_5) \\
&\phantom{=}+ V(-f_1,\phantom{-}f_2,-f_3) \cdot (\phantom{-}U_{eoo}(\f)\,\e_6 - U_{ooo}(\f)\,\e_7), 
\end{align*} where \begin{align*}
U = U_{eee} - U_{oee}\e_1 - U_{eoe}\e_2 + U_{ooe}\e_3 - U_{eeo}\e_4 + U_{oeo}\e_5 + U_{eoo}\e_6 - U_{ooo}\e_7
\end{align*} is a~sum of $8$ terms with different parity with relation to $x_1$, $x_2$, and $x_3$, as in~\eqref{eq:3_20a}--\eqref{eq:3_20h}. 
\end{theorem}
\medskip

This theorem is the generalization of results presented in~\cite{bulow} and~\cite{Ell1993}. Moreover, if at least one of the functions $u$ or $v$ is even with respect to both $x_1$ and~$x_2$ then the abovementioned formula reduces to the well-known form. However, one should bear in mind the fact that in general the above complicated form this claim is of little use. In the next section we will provide the argument similar to one presented in~\cite{Ell1993} which will give much simpler formula.
\medskip

\begin{proof} [of Theorem~\ref{the:3_4}]
We will use the fact that for every $\alpha_1,\alpha_2,\alpha_3\in\R$ we have \begin{align}
\big((e^{-\e_1\alpha_1}\phantom{\,\,\,\cdot\e_1})\cdot(e^{-\e_2\alpha_2}\phantom{\,\,\,\cdot\e_2})\big)\cdot(e^{-\e_4\alpha_3} \phantom{\,\,\,\cdot\e_4}) &= \big((e^{-\e_1\alpha_1}\cdot e^{-\e_2\alpha_2})\cdot e^{-\e_4\alpha_3} \big), \label{eq:7_6a}\\
\big((e^{-\e_1\alpha_1}\cdot\e_1)\cdot(e^{-\e_2\alpha_2}\phantom{\,\,\,\cdot\e_2})\big)\cdot(e^{-\e_4\alpha_3}\phantom{\,\,\,\cdot\e_4}) &= \big((e^{-\e_1\alpha_1}\cdot e^{\phantom{-}\e_2\alpha_2})\cdot e^{\phantom{-}\e_4\alpha_3} \big)\cdot\e_1, \\
\big((e^{-\e_1\alpha_1}\phantom{\,\,\,\cdot\e_1})\cdot(e^{-\e_2\alpha_2}\cdot\e_2)\big)\cdot(e^{-\e_4\alpha_3}\phantom{\,\,\,\cdot\e_4}) &= \big((e^{-\e_1\alpha_1}\cdot e^{-\e_2\alpha_2})\cdot e^{\phantom{-}\e_4\alpha_3} \big)\cdot\e_2, \\
\big((e^{-\e_1\alpha_1}\cdot\e_1)\cdot(e^{-\e_2\alpha_2}\cdot\e_2)\big)\cdot(e^{-\e_4\alpha_3}\phantom{\,\,\,\cdot\e_4}) &= \big((e^{-\e_1\alpha_1}\cdot e^{\phantom{-}\e_2\alpha_2})\cdot e^{\phantom{-}\e_4\alpha_3} \big)\cdot\e_3, \\
\big((e^{-\e_1\alpha_1}\phantom{\,\,\,\cdot\e_1})\cdot(e^{-\e_2\alpha_2}\phantom{\,\,\,\cdot\e_2})\big)\cdot(e^{-\e_4\alpha_3}\cdot\e_4) &= \big((e^{-\e_1\alpha_1}\cdot e^{-\e_2\alpha_2})\cdot e^{-\e_4\alpha_3} \big)\cdot\e_4, \\
\big((e^{-\e_1\alpha_1}\cdot\e_1)\cdot(e^{-\e_2\alpha_2}\phantom{\,\,\,\cdot\e_2})\big)\cdot(e^{-\e_4\alpha_3}\cdot\e_4) &= \big((e^{-\e_1\alpha_1}\cdot e^{-\e_2\alpha_2})\cdot e^{\phantom{-}\e_4\alpha_3} \big)\cdot\e_5, \\
\big((e^{-\e_1\alpha_1}\phantom{\,\,\,\cdot\e_1})\cdot(e^{-\e_2\alpha_2}\cdot\e_2)\big)\cdot(e^{-\e_4\alpha_3}\cdot\e_4) &= \big((e^{\phantom{-}\e_1\alpha_1}\cdot e^{-\e_2\alpha_2})\cdot e^{\phantom{-}\e_4\alpha_3} \big)\cdot\e_6, \\
\big((e^{-\e_1\alpha_1}\cdot\e_1)\cdot(e^{-\e_2\alpha_2}\cdot\e_2)\big)\cdot(e^{-\e_4\alpha_3}\cdot\e_4) &= \big((e^{\phantom{-}\e_1\alpha_1}\cdot e^{-\e_2\alpha_2})\cdot e^{\phantom{-}\e_4\alpha_3} \big)\cdot\e_7. \label{eq:7_6h}
\end{align}

From the definition of the OFT and the convolution, the Fubini theorem and using integration by substitution we have \begin{align*}
&\!\!\!\!\!\!\!\!\!\!\!\!\!\!\!\!\!\!\!\!\int_{\R^3} \left(\int_{\R^3} u(\y)\cdot v(\x-\y)\,\de\y\right)\cdot e^{-\e_1 2\pi f_1x_1}e^{-\e_2 2\pi f_2x_2}e^{-\e_4 2\pi f_3x_3}\,\de \x \\
&=
\int_{\R^3} u(\y)\cdot \left(\int_{\R^3} v(\x-\y)\cdot e^{-\e_1 2\pi f_1x_1}e^{-\e_2 2\pi f_2x_2}e^{-\e_4 2\pi f_3x_3}\,\de \x\right)\,\de\y \\
&=
\int_{\R^3} u(\y)\cdot \left(\int_{\R^3} v(\x)\cdot e^{-\e_1 2\pi f_1(x_1+y_1)}e^{-\e_2 2\pi f_2(x_2+y_2)}e^{-\e_4 2\pi f_3(x_3+y_3)}\,\de \x\right)\,\de\y = (\star).
\end{align*} Consider the inner integral. We can write the transformation kernel in the following way: \begin{align*}
&\!\!\!\!\!\!\!\!\!\!\!\!\!\!\!\!\!\!\!\!e^{-\e_1 2\pi f_1(x_1+y_1)}e^{-\e_2 2\pi f_2(x_2+y_2)}e^{-\e_4 2\pi f_3(x_3+y_3)} \\
=\phantom{+}\,& 
\big((e^{-\e_1\alpha_1}\cdot\phantom{\e_1}\cos(\beta_1))\cdot(e^{-\e_2\alpha_2}\cdot\phantom{\e_2}\cos(\beta_2))\big)\cdot(e^{-\e_4\alpha_3}\cdot\phantom{\e_4}\cos(\beta_3)) \displaybreak[3]\\
-\,&\big((e^{-\e_1\alpha_1}\cdot\e_1\sin(\beta_1))\cdot(e^{-\e_2\alpha_2}\cdot\phantom{\e_2}\cos(\beta_2))\big)\cdot(e^{-\e_4\alpha_3}\cdot\phantom{\e_4}\cos(\beta_3)) \displaybreak[3]\\
-\,&\big((e^{-\e_1\alpha_1}\cdot\phantom{\e_1}\cos(\beta_1))\cdot(e^{-\e_2\alpha_2}\cdot\e_2\sin(\beta_2))\big)\cdot(e^{-\e_4\alpha_3}\cdot\phantom{\e_4}\cos(\beta_3)) \displaybreak[3]\\
+\,&\big((e^{-\e_1\alpha_1}\cdot\e_1\sin(\beta_1))\cdot(e^{-\e_2\alpha_2}\cdot\e_2\sin(\beta_2))\big)\cdot(e^{-\e_4\alpha_3}\cdot\phantom{\e_4}\cos(\beta_3)) \displaybreak[3]\\
-\,&\big((e^{-\e_1\alpha_1}\cdot\phantom{\e_1}\cos(\beta_1))\cdot(e^{-\e_2\alpha_2}\cdot\phantom{\e_2}\cos(\beta_2))\big)\cdot(e^{-\e_4\alpha_3}\cdot\e_4\sin(\beta_3)) \displaybreak[3]\\
+\,&\big((e^{-\e_1\alpha_1}\cdot\e_1\sin(\beta_1))\cdot(e^{-\e_2\alpha_2}\cdot\phantom{\e_2}\cos(\beta_2))\big)\cdot(e^{-\e_4\alpha_3}\cdot\e_4\sin(\beta_3)) \displaybreak[3]\\
+\,&\big((e^{-\e_1\alpha_1}\cdot\phantom{\e_1}\cos(\beta_1))\cdot(e^{-\e_2\alpha_2}\cdot\e_2\sin(\beta_2))\big)\cdot(e^{-\e_4\alpha_3}\cdot\e_4\sin(\beta_3)) \displaybreak[3]\\
-\,&\big((e^{-\e_1\alpha_1}\cdot\e_1\sin(\beta_1))\cdot(e^{-\e_2\alpha_2}\cdot\e_2\sin(\beta_2))\big)\cdot(e^{-\e_4\alpha_3}\cdot\e_4\sin(\beta_3)) ,
\end{align*} where $\alpha_i=2\pi f_ix_i$, $\beta_i=2\pi f_iy_i$, $i=1,2,3$.
\medskip

Using equations~\eqref{eq:7_6a}--\eqref{eq:7_6h} we get \begin{align*}
(\star)&= \phantom{+}V(\phantom{-}f_1,\phantom{-}f_2,\phantom{-}f_3)\cdot U_{eee}(\f)\phantom{\e_1}-V(\phantom{-}f_1,-f_2,-f_3)\cdot U_{oee}(\f)\e_1 \\
&\phantom{=}-V(\phantom{-}f_1,\phantom{-}f_2,-f_3)\cdot U_{eoe}(\f)\e_2 + V(\phantom{-}f_1,-f_2,-f_3)\cdot U_{ooe}(\f)\e_3 \\
&\phantom{=}-V(\phantom{-}f_1,\phantom{-}f_2,\phantom{-}f_3)\cdot U_{eeo}(\f)\e_4 +V(\phantom{-}f_1,\phantom{-}f_2,-f_3)\cdot U_{oeo}(\f)\e_5 \\
&\phantom{=}+V(-f_1,\phantom{-}f_2,-f_3)\cdot U_{eoo}(\f)\e_6 - V(-f_1,\phantom{-}f_2,-f_3)\cdot U_{ooo}(\f)\e_7,
\end{align*} which, after rearranging the terms, concludes the proof.
\end{proof}

Note that (due to the commutativity of convolution) the following formula is also valid: \begin{align*}
\FourierX{OFT}{u*v}(\f) &=\phantom{+} U(\phantom{-}f_1,\phantom{-}f_2,\phantom{-}f_3) \cdot (\phantom{-}V_{eee}(\f)\phantom{\,\e_1} - V_{eeo}(\f)\,\e_4) \nonumber\\
&\phantom{=}+ U(\phantom{-}f_1,-f_2,-f_3) \cdot (-V_{oee}(\f)\,\e_1 + V_{ooe}(\f)\,\e_3) \nonumber\\
&\phantom{=}+ U(\phantom{-}f_1,\phantom{-}f_2,-f_3) \cdot (-V_{eoe}(\f)\,\e_2 + V_{oeo}(\f)\,\e_5) \nonumber\\
&\phantom{=}+ U(-f_1,\phantom{-}f_2,-f_3) \cdot (\phantom{-}V_{eoo}(\f)\,\e_6 - V_{ooo}(\f)\,\e_7),
\end{align*} where \begin{align*}
V = V_{eee} - V_{oee}\e_1 - V_{eoe}\e_2 + V_{ooe}\e_3 - V_{eeo}\e_4 + V_{oeo}\e_5 + V_{eoo}\e_6 - V_{ooo}\e_7
\end{align*} is a~sum of $8$ terms with different parity.
\medskip 

At the end of this section, we will cite several other results that are important from the point of view of system analysis, i.e. octonion analogues of Parseval-Plancherel Theorems for real-valued functions, which we proved in~\cite{BlaszczykSnopek2017}. 

%
%

\begin{theorem} \label{the:6_11}
Let $u,v\colon\R^3\to\R$ be square-integrable functions (in Lebesgue sense). Then \begin{equation*} 
\innerL{u}{v} = \innerL{U_{\mathrm{OFT}}}{V_{\mathrm{OFT}}},
\end{equation*} where \begin{align*}
\innerL{f}{g} = \int_{\R^3} f(\x)\cdot\conj{g}(\x)\,\de\x
\end{align*} denotes the classical scalar product of functions (real- or octonion-valued) of $3$ variables. 
\end{theorem}

%

In~\cite{BlaszczykEUSIPCO2018} we presented a detailed commentary on these assertions, including indicating the significance of the assumption in Theorem~\ref{the:6_11} that the considered functions are real-valued -- for the octonion-valued functions the claim of Theorem~\ref{the:6_11} doesn't hold. In case of real-valued functions Theorem~\ref{the:6_12} (also known in classical theory as Rayleigh Theorem) is direct corollary of Theorem~\ref{the:6_11}, but (as we proved in~\cite{BlaszczykEUSIPCO2018} and was shown independently in~\cite{Lian2019}) is valid also in the general case of octonion-valued functions.

\begin{theorem} \label{the:6_12}
$L^2$-norm of any function $u\colon\R^3\to\O$ (square-integrable in Lebesgue sense) is equal to the $L^2$-norm of its octonion Fourier transform, i.e. \begin{equation*} 
\norm{u}_{L^2(\R^3)} = \norm{U_{\mathrm{OFT}}}_{L^2(\R^3)},
\end{equation*} where $
\norm{f}_{L^2(\R^3)} = \left(\int_{\R^3} \abs{f(\x)}^2\,\de\x\right)^{1/2}
$ for any square-integrable function $f\colon\R^3\to\O$.
\end{theorem}

Of course, the above theorem shows that OFT preserves the energy of octonion-valued functions. However, it is worth mentioning the recent result in~\cite{Lian2019}, where the author argues that OFT of octonion-valued function also satisfies the Hausdorff-Young inequality.

\section{Multidimensional linear time-invariant systems}
\label{sec:5}

As we mentioned in the previous section, the formulas in the theorems on the OFT properties are quite complicated and it seems that they can not be applied in practice. In this section we will show that using the notation of quadruple-complex numbers, we can simplify these expressions.
\medskip

We will focus on using the OFT and notion of quadruple-complex numbers in the analysis of 3-D linear time-invariant (LTI) systems of linear partial differential equations (PDEs) with constant coefficients. The classical Fourier transform is well recognized tool in solving linear PDEs with constant coefficients due to the fact, that it reduces differential equations into algebraic equations~\cite{Allen2004}. It is true also in case of the quaternion Fourier transform~\cite{Ell1993} and, as we will present in this section, the octonion Fourier transform. We have already signaled the possibility of this application in~\cite{BlaszczykECMI2018}, here we will develop these considerations and show additional examples.
\medskip

In Section~\ref{sec:3} we derived formulas for the OFT of partial derivatives. We can now reformulate (by straightforward computations) those formulas using the multiplication in $\EE$ algebra.
\medskip

\begin{corollary}
Let $u\colon\R^3\to\R$ and~$U=\FourierX{OFT}{u}$. Then \begin{align*}
U^{\partial x_1}(\f) &= U(\f)\odot(2\pi f_1)\e_1, \\
U^{\partial x_2}(\f) &= U(\f)\odot(2\pi f_2)\e_2, \\
U^{\partial x_1 x_2}(\f) &= U(\f)\odot(2\pi f_1)(2\pi f_2) \e_3, \\
U^{\partial x_3}(\f) &= U(\f)\odot(2\pi f_3)\e_4, \\
U^{\partial x_1 x_3}(\f) &= U(\f)\odot(2\pi f_1)(2\pi f_3)\e_5, \\
U^{\partial x_2 x_3}(\f) &= U(\f)\odot(2\pi f_2)(2\pi f_3)\e_6, \\
U^{\partial x_1 x_2 x_3}(\f) &= U(\f)\odot(2\pi f_1)(2\pi f_2)(2\pi f_3)\e_7.
\end{align*}
\end{corollary}
\medskip

It is worth noting here the advantages that the above theorem on the partial derivatives transform brings. Let $u\colon\R^3\to\R$ be a~function even with respect to each variable. Both classic and octonion Fourier transforms of $u$ are real-valued functions. Using the classical Fourier transform for the function $u_{x_1x_2}$ we get $-U(\f)\cdot(2\pi f_1)(2\pi f_2)$, and thus also the real-valued function. In a sense, we lose information that the function has been differentiated at all.
In turn, OFT of function $u_{x_1x_2}$ is equal to $U(\f)\cdot(2\pi f_1)(2\pi f_2)\e_3$, therefore it is a~purely imaginary function (only the imaginary part standing next to the unit $\e_3$ will be non-zero). This clearly indicates differentiation with respect to variables $x_1$ and $x_2$.
\medskip

Every linear partial differential equation with constant coefficients can be reduced to algebraic equation (with respect to multiplication in $\mathbb{F}$). Note that in the case of second order equations in which there are no mixed derivatives, this is also true in the sense of multiplication of octonions, e.g. for a~nonhomogeneous wave equation, i.e. \begin{align*}
u_{tt} = u_{x_1x_1} + u_{x_2x_2} + f(t,x_1,x_2)
\end{align*} we have \begin{align*}
\big((2\pi f_1)^2 + (2\pi f_2)^2 -(2\pi\tau)^2 \big) \cdot U(\tau,f_1,f_2) = F(\tau,f_1,f_2),
\end{align*} where $U=\FourierX{OFT}{u}$ and $F=\FourierX{OFT}{f}$. But on the other hand, if we consider the heat equation, i.e. \begin{align*}
u_{t} = u_{x_1x_1} + u_{x_2x_2} + f(t,x_1,x_2),
\end{align*} where we get \begin{align*}
\big((2\pi f_1)^2 + (2\pi f_2)^2 + (2\pi\tau)\e_1 \big) \odot U(\tau,f_1,f_2) = F(\tau,f_1,f_2).
\end{align*} An inverse (in sense of multiplication in $\mathbb{F}$) of $\big((2\pi f_1)^2 + (2\pi f_2)^2 + (2\pi\tau)\e_1 \big)$ exists if and only if $(\tau,f_1,f_2)\neq (0,0,0)$ and is equal to \begin{align*}
\big((2\pi f_1)^2 + (2\pi f_2)^2 + (2\pi\tau)\e_1 \big)^{-1} = \frac{(2\pi f_1)^2 + (2\pi f_2)^2 - (2\pi\tau)\e_1}{\big((2\pi f_1)^2 + (2\pi f_2)^2\big)^2 + (2\pi\tau)^2}.
\end{align*} Hence \begin{align*}
U(\tau,f_1,f_2) = \frac{(2\pi f_1)^2 + (2\pi f_2)^2 - (2\pi\tau)\e_1}{\big((2\pi f_1)^2 + (2\pi f_2)^2\big)^2 + (2\pi\tau)^2}\odot F(\tau,f_1,f_2).
\end{align*} You can not get such a simple formula using multiplication in octonion algebra. Additional theoretical considerations regarding partial differential equations and the use of integral transforms in Cayley-Dickson algebras can be found in \cite{Ludkovsky2010}.
\medskip 

Moreover, the notion of quadruple-complex number multiplication can be used to describe general linear time-invariant systems of three variables and to reduce parallel, cascade and feedback connections of linear systems into simple algebraic equations, as in classical system theory.
\medskip

Consider a~3-D LTI system. We know from the classical signal and system theory that it can be described by its impulse response $h\colon\R^3\to\R$ (sometimes called its Green function) and then, given the input signal $u\colon\R^3\to\R$, the output $v\colon\R^3\to\R$ of such system is given by the formula: \begin{equation*}
v(\x) = \int_{\R^3} u(\y)\cdot h(\x-\y)\,\de\y = (u*h)(\x).
\end{equation*} The output is the convolution of the input signal and the impulse response of the system, which can be schematicaly presented as in Fig.~\ref{fig:7_1}.
\medskip

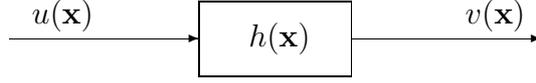
\begin{figure}[t]
\setlength{\unitlength}{10mm}
\begin{center}
\begin{picture}(8.,1.2)
\linethickness{0.075mm}
\put(0.5,0.5){\vector(1,0){2.5}} \put(0.8,0.7){$u(\x)$}

\put(3.0,0.0){\line(0,1){1.0}}
\put(5.0,0.0){\line(0,1){1.0}}
\put(3.0,0.0){\line(1,0){2.0}}
\put(3.0,1.0){\line(1,0){2.0}}

\put(5.0,0.5){\vector(1,0){2.5}}
\put(6.5,0.7){$v(\x)$}
\put(3.65,0.4){$h(\x)$}
\end{picture}

\end{center}
\caption{3-D LTI system.}
\label{fig:7_1}
\end{figure}

\begin{corollary}
The dependence between the OFTs of an input $u$ and an output $v$ of the 3-D LTI system with the impulse response $h$ is given by\begin{align}
V(f_1,f_2,f_3)&=\phantom{+} H_{\mathrm{OFT}}(\phantom{-}f_1,\phantom{-}f_2,\phantom{-}f_3) \cdot (\phantom{-}U_{eee}(\f)\phantom{\,\e_1} - U_{eeo}(\f)\,\e_4) \nonumber\\
&\phantom{=}+ H_{\mathrm{OFT}}(\phantom{-}f_1,-f_2,-f_3) \cdot (-U_{oee}(\f)\,\e_1 + U_{ooe}(\f)\,\e_3) \nonumber\\
&\phantom{=}+ H_{\mathrm{OFT}}(\phantom{-}f_1,\phantom{-}f_2,-f_3) \cdot (-U_{eoe}(\f)\,\e_2 + U_{oeo}(\f)\,\e_5) \nonumber\\
&\phantom{=}+ H_{\mathrm{OFT}}(-f_1,\phantom{-}f_2,-f_3) \cdot (\phantom{-}U_{eoo}(\f)\,\e_6 - U_{ooo}(\f)\,\e_7), \label{eq:7_10}
\end{align} where $V=\FourierX{OFT}{v}$, $U=\FourierX{OFT}{u}$ and $H_{\mathrm{OFT}}=\FourierX{OFT}{h}$ will be called \textit{the octonion frequency response} of the system.
\end{corollary}

\begin{corollary}
Formula~\eqref{eq:7_10} can be restated using the multiplication in $\EE$ algebra. We have \begin{align*}
V(\f) = H_{\mathrm{OFT}}(\f)\odot U(\f)
\end{align*} which is the same as the very well known input-output relation in classic signal and system theory.
\end{corollary}

Consider now the classical connections between the systems, i.e. cascade, parallel and feedback connections. We start with the cascade connection (Fig.~\ref{fig:7_3}), for which we can write \begin{align*}
V(\f) = H_{2,\mathrm{OFT}}(\f)\odot W(\f), \\
W(\f) = H_{1,\mathrm{OFT}}(\f)\odot U(\f),
\end{align*} where $V=\FourierX{OFT}{v}$, $W=\FourierX{OFT}{w}$, $U=\FourierX{OFT}{u}$, $H_{1,\mathrm{OFT}}=\FourierX{OFT}{h_1}$ \linebreak and $H_{2,\mathrm{OFT}}=\FourierX{OFT}{h_2}$. Since the multiplication in $\mathbb{F}$ is commutative and alternative, we obtain \begin{align*}
V(\f) = H_{\mathrm{OFT}}(\f)\odot U(\f),\quad\text{where }\,
H_{\mathrm{OFT}}(\f) = H_{1,\mathrm{OFT}}(\f)\odot H_{2,\mathrm{OFT}}(\f),
\end{align*} as we have for classical (complex) Fourier transform.
\medskip

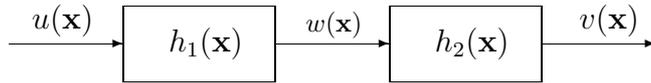
\begin{figure}[h]
\setlength{\unitlength}{10mm}
\begin{center}
\begin{picture}(9.5,1.2)
\linethickness{0.075mm}
\put(0.5,0.5){\vector(1,0){1.5}} \put(0.8,0.7){$u(\x)$}

\put(2.0,0.0){\line(0,1){1.0}}
\put(4.0,0.0){\line(0,1){1.0}}
\put(2.0,0.0){\line(1,0){2.0}}
\put(2.0,1.0){\line(1,0){2.0}}
\put(2.6,0.4){$h_1(\x)$}

\put(4.0,0.5){\vector(1,0){1.5}} \put(4.4,0.66){\footnotesize $w(\x)$}

\put(5.5,0.0){\line(0,1){1.0}}
\put(7.5,0.0){\line(0,1){1.0}}
\put(5.5,0.0){\line(1,0){2.0}}
\put(5.5,1.0){\line(1,0){2.0}}
\put(6.1,0.4){$h_2(\x)$}

\put(7.5,0.5){\vector(1,0){1.5}} \put(8.0,0.7){$v(\x)$}
\end{picture}
\end{center}
\caption{Cascade connection of 3-D LTI systems.}
\label{fig:7_2}
\end{figure}

In the case of parallel connection (Fig.~\ref{fig:7_3}) the computations are much simpler. We get \begin{align*}
V(\f) = H_{\mathrm{OFT}}(\f)\odot U(\f),\quad\text{where }\,
H_{\mathrm{OFT}}(\f) = H_{1,\mathrm{OFT}}(\f)+ H_{2,\mathrm{OFT}}(\f),
\end{align*} the same as in the classical theory.
\medskip

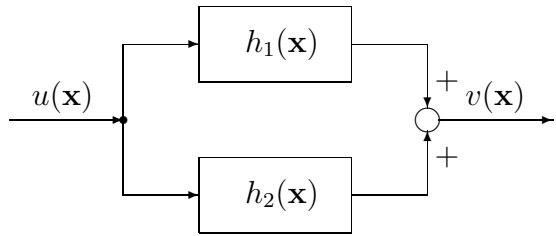
\begin{figure}[h]
\setlength{\unitlength}{10mm}
\begin{center}
\begin{picture}(8.,3.0)
\linethickness{0.075mm}
\put(0.5,1.5){\vector(1,0){1.5}} \put(0.8,1.7){$u(\x)$}
\put(2.0,1.5){\circle*{0.1}}
\put(2.0,0.5){\line(0,1){2.0}}
\put(2.0,2.5){\vector(1,0){1.0}}
\put(3.0,2.0){\line(0,1){1.0}}
\put(5.0,2.0){\line(0,1){1.0}}
\put(3.0,2.0){\line(1,0){2.0}}
\put(3.0,3.0){\line(1,0){2.0}}
\put(5.0,2.5){\line(1,0){1.0}}
\put(6.0,2.5){\vector(0,-1){0.85}}
\put(2.0,0.5){\vector(1,0){1.0}}
\put(3.0,0.0){\line(0,1){1.0}}
\put(5.0,0.0){\line(0,1){1.0}}
\put(3.0,0.0){\line(1,0){2.0}}
\put(3.0,1.0){\line(1,0){2.0}}
\put(5.0,0.5){\line(1,0){1.0}}
\put(6.0,0.5){\vector(0,1){0.85}}
\put(6.0,1.5){\circle{0.3}}
\put(6.15,1.5){\vector(1,0){1.5}}
\put(6.5,1.7){$v(\x)$}
\put(3.6,2.4){$h_1(\x)$}
\put(3.6,0.4){$h_2(\x)$}
\put(6.1,1.9){$+$}
\put(6.1,0.9){$+$}
\end{picture}
\end{center}
\caption{Parallel connection of 3-D LTI systems.}
\label{fig:7_3}
\end{figure}

It gets more complicated if we consider the feedback connection, as in~Fig.\ref{fig:7_4}. We can write the system of equations: \begin{align*}
V(\f) &= H_{1,\mathrm{OFT}}(\f)\odot W(\f), \\
W(\f) &= U(\f) - H_{2,\mathrm{OFT}}(\f)\odot V(\f).
\end{align*} Using the commutativity and associativity of $\odot$ we get \begin{align*}
(1 + H_{1,\mathrm{OFT}}(\f)\odot H_{2,\mathrm{OFT}}(\f)) \odot V(\f) = H_{1,\mathrm{OFT}}(\f)\odot U(\f),
\end{align*} which leads to \begin{align*}
V(\f) = H_{\mathrm{OFT}}(\f)\odot U(\f),\quad\text{where }\, H_{\mathrm{OFT}}(\f) = (1 + H_{1,\mathrm{OFT}}(\f)\odot H_{2,\mathrm{OFT}}(\f))^{-1} \odot H_{1,\mathrm{OFT}}(\f),
\end{align*} and the inverse is in sense of multiplication in~$\mathbb{F}$. The formula is very well-known, but we need to remember that not every element of $\mathbb{F}$ has its inverse. Hence we can see that not every 3-D LTI system can be described with the convolution formula and analyzed with the OFT.
\medskip

\begin{figure}[h]
\setlength{\unitlength}{10mm}
\begin{center}
\begin{picture}(8.,3.0)
\linethickness{0.075mm}
\put(0.5,2.5){\vector(1,0){1.35}} \put(0.8,2.7){$u(\x)$}
\put(2.0,2.5){\circle{0.3}}
\put(2.15,2.5){\vector(1,0){0.85}}
\put(2.0,0.5){\vector(0,1){1.85}}
\put(2.0,0.5){\line(1,0){1.0}}
\put(1.7,2.7){$+$}
\put(2.1,2.1){$-$}
\put(2.2,2.65){\footnotesize $w(\x)$}

\put(3.0,2.0){\line(0,1){1.0}}
\put(5.0,2.0){\line(0,1){1.0}}
\put(3.0,2.0){\line(1,0){2.0}}
\put(3.0,3.0){\line(1,0){2.0}}

\put(5.0,2.5){\vector(1,0){1.0}}
\put(6.0,2.5){\circle*{0.1}}
\put(6.0,0.5){\line(0,1){2.0}}
\put(6.0,0.5){\vector(-1,0){1.0}}

\put(3.0,0.0){\line(0,1){1.0}}
\put(5.0,0.0){\line(0,1){1.0}}
\put(3.0,0.0){\line(1,0){2.0}}
\put(3.0,1.0){\line(1,0){2.0}}

\put(6.0,2.5){\vector(1,0){1.5}}
\put(6.5,2.7){$v(\x)$}
\put(3.6,2.4){$h_1(\x)$}
\put(3.6,0.4){$h_2(\x)$}
\end{picture}
\end{center}
\caption{Feedback connection of 3-D LTI systems.}
\label{fig:7_4}
\end{figure}
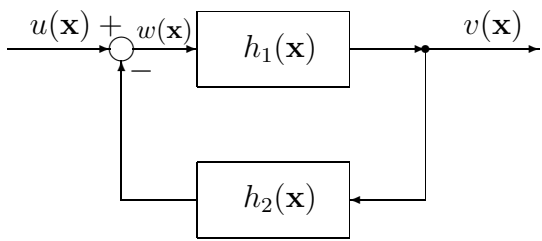

\section{Discussion and conclusions}
\label{sec:6}

It has been shown that the theory of octonion Fourier transforms can be generalized to the case of functions with values in higher-order algebras. Those transforms have properties that are similar to their  classical (complex) counterparts. Octonion analogues of scaling, modulation and shift theorems proved in Section~\ref{sec:3} form the foundation of octonion-based signal and system theory. Properties of the octonion Fourier transform in context of other signal-domain operations, i.e. derivation and convolution of real-valued functions, show that it is a~fine tool in the analysis of multidimensional LTI systems. 
\medskip

It remains to study the discrete case, i.e. discrete-space octonion Fourier transform (DSOFT). Preliminary studies show that the notion of quadruple-complex numbers can be applied to define the DSOFT and to analyze linear difference equations.


\subsubsection*{Acknowledgments}

\noindent The~research conducted by the author was supported by National Science Centre (Poland) grant No. 2016/23/N/ST7/00131.

\bigskip

\bibliographystyle{plain}
\bibliography{bibliography}

\end{document}